\documentclass[12pt]{amsart}
\usepackage{color}
\usepackage{amssymb}
\usepackage{comment}
\usepackage{pdfpages}
\usepackage{graphicx}
\usepackage{dcolumn}
\usepackage{algorithmicx}
\usepackage{algorithm,caption}
\usepackage{algpseudocode}
\usepackage{bbm}

\RequirePackage[numbers]{natbib}
\RequirePackage[colorlinks=true, pdfstartview=FitV, linkcolor=blue,
  citecolor=blue, urlcolor=blue]{hyperref}
\usepackage{paralist}

\usepackage{tikz} 
\usepackage{dcolumn}
\usetikzlibrary{arrows,positioning}
\tikzset{
    >=stealth',
    pil/.style={
           ->,
           thick,
           shorten <=2pt,
           shorten >=2pt,}
}

\usepackage{graphicx}
\graphicspath{{./figures/}}
\usepackage{amsfonts}
\usepackage{amsmath}
\usepackage{amsthm}
\usepackage{amssymb}
\usepackage{amsbsy}
\usepackage{epsfig}
\usepackage{fullpage}
\usepackage{natbib, mathrsfs} 
\usepackage{verbatim}
\usepackage[latin1]{inputenc}
\usepackage{mhequ}
\usepackage{upgreek}
\numberwithin{equation}{section}

\def \be{\begin{equs}}
\def \ee{\end{equs}}

\def \d{\mathrm{d}}

\newtheorem{theorem}{Theorem}[section]
\newtheorem{lemma}[theorem]{Lemma}
\newtheorem{rem}[theorem]{Remark}

\newtheorem{assumption}[theorem]{Assumption}

\theoremstyle{plain}
\newtheorem{thm}{Theorem}
\newtheorem*{thm-non}{Theorem}

\theoremstyle{definition}
\newtheorem{defn}[theorem]{Definition}

\begin{document}

\title[Rapid Mixing of Geodesic Walks]{Rapid Mixing of Geodesic Walks on Manifolds with Positive Curvature}


\author{Oren Mangoubi$^{\ddag}$}
\thanks{$^{\ddag}$omangoubi@gmail.com, 
 \'{E}cole Polytechnique F\'{e}d\'{e}rale de Lausanne (EPFL),
 IC IINFCOM THL3, Station 14, 1015 Lausanne, Switzerland}

\author{Aaron Smith$^{\sharp}$}
\thanks{$^{\sharp}$smith.aaron.matthew@gmail.com, 
   Department of Mathematics and Statistics,
University of Ottawa, 585 King Edward Avenue, Ottawa
ON K1N 7N5, Canada}

\maketitle

\begin{abstract}
We introduce a Markov chain for sampling from the uniform distribution on a Riemannian manifold $\mathcal{M}$, which we call the \textit{geodesic walk}. We prove that the mixing time of this walk on any manifold with positive sectional curvature $C_{x}(u,v)$ bounded both above and below by $0 < \mathfrak{m}_{2} \leq C_{x}(u,v) \leq \mathfrak{M}_2 < \infty$ is $\mathcal{O}^*\left(\frac{\mathfrak{M}_2}{\mathfrak{m}_2}\right)$.  In particular, this bound on the mixing time does not depend explicitly on the dimension of the manifold.  In the special case that $\mathcal{M}$ is the boundary of a convex body, we give an explicit and computationally tractable algorithm for approximating the exact geodesic walk. As a consequence, we obtain an algorithm for sampling uniformly from the surface of a convex body that has running time bounded solely in terms of the curvature of the body.
\end{abstract}


\section{Introduction}

Sampling from manifolds has applications to areas such as statistics \cite{bornn2015moment, diaconis2013sampling}, computer graphics \cite{oztireli2010spectral}, optimization \cite{billiards} and systems biology \cite{rhee2014iterative}. For a simple example, manifolds with nonnegative curvature appear in statistical applications as the level sets of log-concave, or more generally quasiconcave, distributions (see Remark \ref{rem:quasiconcave}). Many natural distributions in statistics are quasiconcave or log-concave, including the ``ridge regression" posterior associated with Gaussian priors for logistic regression \cite{durmus2016sampling2, durmus2016sampling}. Existing computer science literature shows that the problem of sampling from manifolds with nonnegative curvature is tightly linked to the problems of sampling from and computing integrals of quasiconcave distributions \cite{Heat, lovasz2007geometry, interior_sampler,  narayanan2008sampling}. The recent paper \cite{diaconis2013sampling} has a broader survey on this sampling problem in various fields.

In this paper we study a simple Markov chain, which we call the \textit{geodesic walk}, that can sample from the uniform distribution of a general Riemannian manifold $\mathcal{M}$. Our main result is that this walk mixes quickly when $\mathcal{M}$ has positive bounded curvature. We also give results concerning an efficient implementation of this walk in the important special case that $\mathcal{M}$ is the boundary of a convex set. 

\begin{rem}\label{rem:quasiconcave}
We give a technical caveat: a level set of a quasiconcave distribution need not be smooth, and in particular may not have well-defined sectional curvature.  However, there is a standard extension of the notion of sectional curvature to this situation, called the Alexandrov curvature. It turns out that the Alexandrov curvature of the level set of a quasiconcave distribution is always nonnegative \cite{alexandrov2005convex}. It is in this sense that we say that the compact level sets of quasiconcave distributions have nonnegative curvature. We now restrict our attention to situations in which the sectional curvature is well-defined, until a brief discussion in Section \ref{sec:discussion}.
\end{rem}

\subsection{Main Results}

The geodesic walk, defined precisely in Algorithm \ref{alg:Geodesic_Walk} below, is a Markov chain $\{X_{i}\}_{i \in \mathbb{N}}$ that is well-defined on any Riemannian manifold $\mathcal{M}$. It evolves by selecting, at each time step $i$, a random tangent vector $U_{i}$ in the tangent space $\mathcal{T}_{X_{i}}$ of $\mathcal{M}$ at $X_{i}$, and then following the associated geodesic for some period of time. The geodesic walk is a natural Markov chain on a manifold, and it is somewhat similar to the well-studied ``ball walk" on a manifold (see \textit{e.g.}  \cite{Ball_Walk}). It is well-known (see Example 7 of \cite{ollivier2009ricci}) that the ball walk mixes rapidly on manifolds with positive curvature. The first main result of this paper, Theorem \ref{thm:mixing}, is that the geodesic walk also mixes rapidly on a manifold with sectional curvature bounded above and below by two constants $0 < \mathfrak{m}_{2} \leq C_{x}(u,v) \leq \mathfrak{M}_2 < \infty$.

Although the ball walk is easy to define and mixes rapidly, it is generally not practical to implement a version of the ball walk with the desired stationary distribution on a computer, and so it does not give rise to practical sampling algorithms (see Example 3C of \cite{diaconis2013sampling}, titled \emph{How not to sample}). In contrast, we show that the geodesic walk automatically has the correct stationary distribution. Our second main result (Theorems \ref{ThmStatExist} and \ref{thm:uniquenessStat}) is that, unlike the ball walk, the geodesic walk has the uniform measure on the manifold as its unique stationary distribution.  

The geodesic walk is also much easier to implement than the ball walk. Our third main result, Theorem \ref{thm:e1}, shows that any reasonably good algorithm for approximating geodesics can be used to simulate an ``approximate" geodesic walk that mixes quickly and has near-uniform stationary measure. We also show that Algorithm \ref{alg:Approx_Manifold}, which is straightforward to implement, gives an efficient ``approximate" geodesic walk in this sense.

We believe our $\mathcal{O}^*\left(\frac{\mathfrak{M}_2}{\mathfrak{m}_2}\right)$ bound on the mixing time of the geodesic walk is the first dimensionless mixing time bound for an implementable Markov chain with uniform stationary distribution on bounded-positive curvature manifolds in general, and convex body boundaries in particular.  One reason our bound does not depend on the dimension is that, in contrast to previously-proposed algorithms, the geodesic walk can take long steps whose length is independent of the dimension.  For comparison, the best existing bound that we know of is given for the ``billiards walk" studied in \cite{billiards}, which is shown to have mixing time of $\mathcal{O}^*\left(d^2\cdot \frac{\mathfrak{M}_2}{\mathfrak{m}_2}\right)$  (See Remark \ref{rem:billiards}).
That walk cannot mix quickly in high dimensions, even if $\mathcal{M}$ is the surface of a sphere, because the typical step size is $\mathcal{O}(\frac{1}{\sqrt{d}})$.

In both statistics and computer science, one of the most important and well-studied special cases occurs when $\mathcal{M}$ is the boundary $\partial \mathcal{K}$ of a convex body $\mathcal{K} \subset \mathbb{R}^{d+1}$. There are many Markov chains for sampling from the uniform distribution on the boundary $\partial \mathcal{K}$ of a set under various assumptions (see \textit{e.g.} \cite{belisle1993hit}, \cite{comets2009billiards}, \cite{billiards}), all of which involve a walk that moves inside of $ \mathcal{K}$ and then reflects off of $\partial \mathcal{K}$.  In Section \ref{sec:approximating}, we explain how our geodesic walk can be well-approximated by such a ``reflecting walk," if the reflecting walk is allowed to have long-term momentum. Our final main result, Theorem \ref{thm:mixing_convex}, shows that this approximate version of the geodesic walk gives us an algorithm for sampling uniformly from the boundary $\partial  \mathcal{K}$ of a convex set $\mathcal{K}$ with arbitrarily small error $\epsilon > 0$ in roughly $\mathcal{O}^*\left((\frac{\mathfrak{M}_{2}}{\mathfrak{m}_{2}})^{3} \epsilon^{-1} \right)$ reflections, provided that $\mathcal{M}$ has inner and outer radii of curvature bounded below and above by $\frac{1}{\sqrt{\mathfrak{M}_2}}$ and $\frac{1}{\sqrt{\mathfrak{m}_2}}$, respectively. We believe the $\mathcal{O}^* \left((\frac{\mathfrak{M}_{2}}{\mathfrak{m}_{2}}) ^{3}\right)$ bound for the implementation of the geodesic walk is often very pessimistic and can be greatly improved in some important special cases (See Remark \ref{rem:smooth} in Section \ref{sec:Approx_Convex}). Furthermore, in situations where one has access to better geodesic approximations, such as when symplectic integrators or exact geodesic integrators are available, our results imply computational costs that are much closer to the mixing time $\mathcal{O}^*\left(\frac{\mathfrak{M}_2}{\mathfrak{m}_2}\right)$ of the true geodesic walk.

\begin{rem} \label{rem:billiards}
Note that the bound in \cite{billiards} is better for chains with a bound on the diameter of $\mathcal{K}$ that is much smaller than $\frac{2}{\sqrt{\mathfrak{m}_2}}$, while our bound is better in higher dimensions and when the bound on the inner radius of curvature of $\mathcal{K}$ is much smaller than $\frac{1}{\sqrt{\mathfrak{M}_2}}$. Also, a practical implementation of the geodesic walk requires additional computational costs that depend on the method used to approximate geodesics.
\end{rem}

\subsection{Proof Techniques}
The main component of our proof is a coupling argument, where we couple the initial velocity by parallel transport (Section \ref{sec:coupling}).  We then use comparison theorems from differential geometry \cite{abresch1997injectivity, cheeger2008comparison, rauch1951contribution} to show that our assumptions of positive curvature bounds imply that the distance between the two chains contracts over each step in the Markov chain (Section \ref{sec:contraction}).  This contraction estimate immediately implies a bound on the Wasserstein mixing time and other relevant quantities \cite{ollivier2009ricci} (Section \ref{sec:mixing}).  

\begin{rem} [Short and Long Steps]
The geodesic walk has an important tuning parameter: the amount of time $T$ to follow the geodesic. If $T$ is very small, then we can couple two copies $\{X_{i},Y_{i}\}_{i \in \mathbb{N}}$ via parallel transport of the momentum vector $U_{i}$ used at each step (see Definition \ref{DefCouplingConstruction} for a precise definition) to obtain the following contraction estimate (see Proposition 6 of \cite{ollivier2009ricci}):
\begin{equation*}
d(X_{i+1},Y_{i+1}) \leq (1 - \frac{T^{2}}{2} \mathfrak{m}_{2} + O(T^{3} + T^{2} d(X_{i},Y_{i}))) \, d(X_{i},Y_{i}).
\end{equation*}

In other words, for very small step size $T$, a lower bound on the sectional curvature implies that the geodesic walk contracts. This bound does not (directly) imply useful bounds for the geodesic walk with $T$ large, since the error term in this inequality may become very large.
\end{rem}

After studying the mixing properties of the geodesic walk, we study a possible implementation. Using comparison bounds for Markov chains, we show that one can  approximate geodesic trajectories with arbitrary accuracy in a dimension-independent number of steps provided one has access to an appropriate oracle, allowing one to generate samples from a stationary distribution that is arbitrarily close to uniform (in Wasserstein transportation distance) in a dimension-independent number of oracle calls.  We also show how to construct such an oracle in the special case where $\mathcal{M}$ is the boundary of a convex body (Section \ref{sec:approximating}). In the same special case, we explain how to use pre- and post-processing steps to replace a general manifold $\mathcal{M}$ with a manifold $\mathcal{M}'$ with sectional curvature bounded by $\mathcal{O}(d^{2})$ (Section \ref{sec:discussion}), allowing us to sample efficiently even from ``pointy" manifolds that have very high sectional curvature at some locations.

We mention that \cite{tat2016geodesic} obtains a bound on the mixing time of a  related (but different) geodesic walk. Although our Markov chains are somewhat similar, our settings are quite different: their bounds require $\mathcal{M}$ to be the boundary of a convex polytope, and are best when the polytope has few faces; our results require $\mathcal{M}$ to have curvature that is bounded away from both 0 and infinity. We also note that the mixing time bound for the geodesic walk in \cite{tat2016geodesic} is not dimension-independent.   A final difference is that the authors of \cite{tat2016geodesic} state that their proof follows the conductance-based approach that is popular for analyzing geometric random walks \cite{vempala2005geometric}, while our proof uses a probabilistic coupling argument as introduced in \cite{dobrushin70contraction}. See recent work such as \cite{ollivier2009ricci}, \cite{ollivier2011visual} for more on the relationship between couplings of Markov chains and curvature, and also work such as \cite{joulin2010curvature}, \cite{paulin2015concentration} for some consequences of the existence of such couplings, including various concentration inequalities.

\subsection{List of Sections} 
The rest of the paper is arranged as follows:
\begin{itemize}
\item In Section \ref{sec:assumptions} we go over the assumptions we make about $\mathcal{M}$ and Riemannian geometry preliminaries.
\item In Section \ref{sec:Walk} we define the geodesic walk on the manifold $\mathcal{M}$.
\item In Section \ref{sec:Stationary}, we prove that the stationary distribution of the geodesic walk is uniform on $\mathcal{M}$.
\item In Section \ref{sec:coupling} we define a coupling of two copies of the geodesic walk.
\item In Section \ref{sec:contraction} we prove the contraction bound described above.
\item In Section \ref{sec:mixing} we use this contraction bound to prove a bound on the mixing time.
\item In Section \ref{sec:approximating} we show that one can computationally approximate geodesic trajectories to sample the uniform distribution on $\mathcal{M}$ to arbitrary accuracy in a dimension-independent number of steps if one has access to an appropriate oracle (Section \ref{sec:Approx_Manifold}).  We construct such an oracle explicitly in the special case where $\mathcal{M}$ is the boundary of a convex body (Section \ref{sec:Approx_Convex}). 
\item In Section \ref{sec:discussion}, we discuss some future research directions left open by this paper:  In section \ref{sec:discussion_polytope} we give pre- and post-processing steps that allow us to assume the target manifold $\mathcal{M}$ satisfies $\mathfrak{M}_{2} = \mathcal{O}(d^{2})$. This is useful when one wishes to sample from a convex body $\mathcal{K}$ whose boundary does not have bounded curvature, such as the case that $\mathcal{K}$ is a convex polytope.  In section \ref{sec:discussion_HMC} we discuss connections between this paper and recent work \cite{HMC_logconcave} on the mixing times of the popular Hamiltonian Monte Carlo algorithm.
\end{itemize}

\section{Assumptions and Riemannian geometry preliminaries}\label{sec:assumptions}
In this section we recall results from Riemannian geometry and introduce definitions and assumptions that we will use in the rest of the paper.

\begin{itemize}
\item \textbf{Assumptions about our manifold $\mathcal{M}$:} Throughout this paper we will assume that, 
\begin{enumerate}
\item $(\mathcal{M},g)$ is a closed, connected, second-order differentiable Riemannian manifold, with associated inner product $g \equiv g_x$ on the tangent space $\mathcal{T}_x \equiv \mathcal{T}_x \mathcal{M}$ at $x \in \mathcal{M}$. 
\item$(\mathcal{M},g)$ has bounded positive sectional curvature. That is, there exist constants $0 < \mathfrak{m}_2  \leq \mathfrak{M}_2 < \infty$ so that, for all $x\in \mathcal{M}$ and $u, v\in \mathcal{T}_x$, \begin{equation}
\mathfrak{m}_2 \leq C_x(u,v) \leq \mathfrak{M}_{2},
\end{equation}
where $C_x(u,v)$ is the sectional curvature of $\mathcal{M}$ at $x$ in the directions $\frac{u}{\|u\|}$ and $\frac{v}{\|v\|}$.
\end{enumerate}

\item \textbf{Levi-Civita connection:}  The fundamental theorem of Riemannian geometry guarantees the existence and uniqueness of a torsion-free affine connection $\nabla$ on $(\mathcal{M},g)$ that induces an isometry of tangent spaces via parallel transport (see Theorem 6.8 of \cite{RiemannianGeometryNotes}); this connection is called the \textit{Levi-Civita connection}.

\item \textbf{Paths and Metrics:} Recall that the length of a segment of a smooth path $h \, : \, \mathbb{R}^{+} \mapsto \mathcal{M}$ is given by
\begin{equation}
\mathrm{length}(h[a,b]) = \int_a^b \sqrt{g_{h(t)}(h'(t), h'(t))}\mathrm{d}t.
\end{equation}

This gives rise to the metric 
\begin{equation}
\mathrm{dist}(x,y) = \inf_{h \, : \, h(0) = x, \, h(1) = y} \mathrm{length}(h[0,1])
\end{equation}
on $\mathcal{M}$.

\item \textbf{Curvatures:} Recall that the \textit{Riemann curvature tensor} $R$ at a point $x \in \mathcal{M}$ sends a pair of vectors $u,v \in \mathcal{T}_{x}$ to a map $R(u,v) \, : \, \mathcal{T}_{x} \mapsto \mathcal{T}_{x}$ of the tangent space at $x$ to itself, which transforms an element $w \in \mathcal{T}_{x}$ via the formula
\begin{equation*}
R(u,v)[w] = \nabla_{u} \nabla_{v} w - \nabla_{v} \nabla_{u} w  - \nabla_{[u,v]} w, 
\end{equation*} 
where $\nabla$ is the Levi-Civita connection described above and $[\cdot,\cdot]$ is the Lie bracket. This formula allows us to define the \textit{sectional curvature} $R' \, : \, \mathcal{T}_{x}^{2} \mapsto \mathbb{R}$ by 
\begin{equation} \label{EqDefSecCurve}
R'(u,v) = \frac{\langle R(u,v)v, u \rangle}{\langle u,u \rangle \, \langle v,v \rangle - \langle u, v \rangle^{2}},
\end{equation} 
where $\langle \cdot, \cdot \rangle = \langle \cdot, \cdot \rangle_{g}$ is the usual inner product on $\mathcal{T}_{x}$ associated with a Riemannian manifold $(\mathcal{M},g)$.

\item \textbf{Uniform Measures:} We recall two measures, one on the manifold $\mathcal{M}$ and the other on small subsets of the tangent spaces $\mathcal{T}_{x}$, that we will call the \textit{uniform} measures on their associated spaces. First, recall that any Riemannian manifold $(\mathcal{M},g)$ has an associated volume function $\lambda$. If $\lambda(\mathcal{M}) < \infty$, we denote by $\mathrm{Unif}(\mathcal{M})$ the probability measure given by
\begin{equation}
\mathrm{Unif}(\mathcal{M})[A] = \frac{\lambda(A)}{\lambda(\mathcal{M})}
\end{equation}
for $A \subset \mathcal{M}$ measurable. We refer to this as the \textit{uniform measure on $\mathcal{M}$}.

Recall that any choice of the basis $\mathcal{B}$ of the tangent space $\mathcal{T}_{x}$ of a $d$-dimensional manifold $\mathcal{M}$ at a point $x \in \mathcal{M}$ gives a natural isometry $\zeta_{x} =  \zeta_{x, \mathcal{B}} \, : \, \mathcal{T}_{x} \mapsto \mathbb{R}^{d}$. Furthermore, the pull-back of the Lebesgue measure from the unit sphere $\mathbb{S}^{d-1} = \{ x \in \mathbb{R}^{d} \, : \, \| x \|_{2} = 1 \} \subset \mathbb{R}^{d}$ to the set $\mathbb{S}(\mathcal{T}_{x}) = \zeta_{x}^{-1}(\mathbb{S})$ given by $\zeta_{x}$ does not depend on the choice of basis $\mathcal{B}$. We denote by $\mathrm{Unif}(\mathbb{S}(\mathcal{T}_{x}))$ or $\mu$ this unique pull-back of the Lebesgue measure by $\zeta_{x}$, and refer to this measure as the \textit{uniform measure on the unit sphere in $\mathcal{T}_{x}$}. 

Throughout the paper, if $S_{1} \subset S_{2}$ and we have defined a uniform measure on $S_{2}$, we define the uniform measure on $S_{1}$ by
\begin{equation*} 
\mathrm{Unif}(S_{1})[A] = \frac{\mathrm{Unif}(S_{2})[A]}{\mathrm{Unif}(S_{2})[S_{1}]}.
\end{equation*}

\item \textbf{Geodesics:} Define the phase space
\begin{equation} \label{EqPhaseSpaceDef}
\mathcal{M}^\circ := \{(x,v): x \in \mathcal{M}, v \in \mathbb{S}(\mathcal{T}_{x}) \}.
\end{equation}

Associated to every element $(x,v) \in \mathcal{M}^{\circ}$, there is a special path $\gamma_{(x,v)} \, : \, \mathbb{R}^{+} \rightarrow \mathcal{M}$ with $\gamma_{(x,v)}(0) = x$ and $\gamma'_{(x,v)}(0) = v$ that is called the \textit{geodesic}. Roughly speaking, this is the path obtained by starting at point $x$ and traveling with velocity $v$ along the manifold. We will use the following properties of the geodesic \cite{petersen2006riemannian}:
\begin{enumerate}
\item For all $x \in \mathcal{M}$ and $v \in \mathcal{T}_{x}$, there exists $\epsilon = \epsilon(x,v) > 0$ so that 
\begin{equation*} 
\mathrm{length}(\gamma_{(x,v)}[0,t]) = \mathrm{dist}(\gamma_{(x,v)}(0), \gamma_{(x,v)}(t))
\end{equation*} 
for all $0 < t < \epsilon$ (see Section 5 and Exercise 5.9.34 of \cite{petersen2006riemannian}).
\item 
Associated to every smooth map $\tau \, : \, \mathbb{R}^{+} \mapsto \mathcal{M}^{\circ}$, there is an object called the \textit{Jacobi field} 
\begin{equation}
J_\tau(t):= \frac{\partial \gamma_\tau(t)}{\partial \tau}
\end{equation}
that satisfies the Jacobi second-order ordinary differential equation:
\begin{equation}
\frac{\mathrm{D}^2}{\mathrm{d}t^2} J_\tau(t) + R\left(J_\tau(t), \gamma'_\tau(t)\right)\gamma'_\tau(t) = 0,
\end{equation}
where $R$ denotes the Riemann curvature tensor of $\mathcal{M}$ and $\mathrm{D}$ is the covariant derivative with respect to the Levi-Civita connection (See Section 3.2.4 of \cite{petersen2006riemannian}).
\end{enumerate}

\item \textbf{Diameter:} We define $D := \mathrm{Diam}(\mathcal{M})$ to be the diameter of $\mathcal{M}$.  Since the sectional curvature of $\mathcal{M}$ is bounded below by $\mathfrak{m}_2$, the Myers-Bonnet theorem \cite{myers1941riemannian} implies that
\begin{equation} \label{IneqDiameterMyersBonnet}
D \leq \frac{\uppi}{\sqrt{\mathfrak{m}_2}}.
\end{equation}

\item \textbf{Existence of a \emph{unique} parallel transport:} Associated with any \textit{connection} on a manifold and any finite-length path $h \, : \, [a,b] \mapsto \mathcal{M}$ there exists a map from $\mathcal{T}_{h(a)}$ to $\mathcal{T}_{h(b)}$, called the \textit{parallel transport} map. 

We now define the only parallel transport maps that we will use in this paper. The Hopf-Rinow theorem \cite{hopf1931begriff} implies that between any pair of points $x,y \in \mathcal{M}$ there exists at least one unit-speed geodesic $\omega(t) \equiv \omega(t;x,y)$, $t\in [0,\mathrm{dist}(x,y)]$ with $\omega(0;x,y) = x$ and $\omega(\mathrm{dist}(x,y); x,y) = y$ with $\omega(t;x,y) = \omega(\mathrm{dist}(x,y)-t;y,x)$. We choose arbitrarily, via the axiom of choice, a particular family $\{ \omega(\cdot; x,y) \}_{(x,y) \in \mathcal{M}\times \mathcal{M}}$ of such geodesic paths. Finally, we define $\phi(t;\cdot;x,y) \, : \, \mathcal{T}_{x} \mapsto \mathcal{T}_{\omega(t;x,y)}$ to be the parallel transport map associated with the Levi-Civita connection and the path $\omega(\cdot; x,y)$. We also write $\phi(\cdot;x,y) \equiv \phi(\mathrm{dist}(x,y);\cdot;x,y)$ for shorthand.

Finally, fix an arbitrary point $x \in \mathcal{M}$ and an arbitrary basis $\{b_{i}\}$ of $\mathcal{T}_{x}$. Throughout the paper, we denote by $\{ \zeta_{y} \}_{y \in \mathcal{M}}$ the maps $\zeta_{y, \{ \phi(b_{i}; x,y) \}} \, : \, \mathcal{T}_{y} \mapsto \mathbb{R}^{d}$ associated with this basis and its parallel transports by $\phi$.

\item \textbf{Wasserstein distance and mixing time:} For any distribution $\eta$, we write $X \sim \eta$ when the random variable $X$ has distribution $\eta$. For two distributions $\eta,\nu$ on a common measure space $(\Omega,\mathcal{A})$, define $\Xi(\eta,\nu)$ to be the collection of all distributions $\xi$ on $(\Omega^{2}, \mathcal{A} \times \mathcal{A})$ with marginal distributions $\xi(\cdot,\Omega) = \eta(\cdot)$, $\xi(\Omega,\cdot) = \nu(\cdot)$. The Wasserstein transportation distance $W_{d}$ between two measures $\eta$ and $\nu$ on a common metric measure space $(\Omega,d)$ is given by
\[W_{d}(\eta, \nu) = \inf_{(X,Y) \, : \, (X,Y) \in \Xi(\eta,\nu)} \mathbb{E}[d(X,Y)].\]

Consider a transition kernel $K$ on a metric measure space $(\Omega,d)$, with unique stationary distribution $\xi$. We define the Wasserstein mixing profile $\mathfrak{t}_\mathrm{mix} \, : \, [0,1] \mapsto \mathbb{N} $ of $K$ to be
 \[\mathfrak{t}_\mathrm{mix}(\epsilon) = \inf \{\mathfrak{t} \geq 0 \, : \, \sup_{x \in \Omega} \, W_{d}(K^{\mathfrak{t}}(x,\cdot) - \xi(\cdot)) < \epsilon \}. \]

\item \textbf{Big-O Notation:} In this paper we use the ``big-$\mathcal{O}$" notation.  Specifically, for any two functions $f:\mathbb{R} \rightarrow \mathbb{R}$ and $g:\mathbb{R} \rightarrow \mathbb{R}$, we write ``$f(z) = \mathcal{O}(g(z))$" or ``$f(z) \leq \mathcal{O}(g(z))$" if there exists a constant $0<C<\infty$ and some $Z>0$ such that
\[f(z) \leq C \cdot g(z)\]
for all $z>Z$.
Furthermore, write ``$f(z) = \mathcal{O}^*(g(z))$" or ``$f(z) \leq \mathcal{O}^*(g(z))$" if there exist constants $0 < c',C' <\infty$ and some $Z'>0$ such that
\[f(z) \leq C' \cdot g(z)\log(z)^{c'}\]
for all $z>Z'$.  I.e., the ``big-$\mathcal{O}^*$" notation suppresses the logarithmic terms that we would otherwise need to take into account when using the ``big-$\mathcal{O}$" notation.

\item \textbf{Angles:} For two vectors $v_{1}, v_{2} \in \mathbb{R}^{d}$, we denote by $\measuredangle (v_{1},v_{2})$ the angle between the vectors.  For any hyperplanes $\mathbbm{n}_1^\perp, \mathbbm{n}_2^\perp$ with normal vectors $\mathbbm{n}_1$ and $\mathbbm{n}_2$, respectively, we define
\be
\measuredangle(\mathbbm{n}_1^\perp, \mathbbm{n}_2^\perp):= \measuredangle(\mathbbm{n}_1, \mathbbm{n}_2).
\ee
For a plane $P$ of any dimension and any vector $v$, the angle $\measuredangle (v,P)$ between $v$ and $P$ is defined as the angle between $v$ and its projection $\mathrm{proj}_Pv$ onto $P$:
\be
\measuredangle (v,P) \equiv \measuredangle (P,v) := \measuredangle(v, \mathrm{proj}_Pv).
\ee
For any hyperplane $\mathbbm{n}^\perp$ with normal vector $\mathbbm{n}$, we define the angle between $\mathbbm{n}^\perp$ and a plane $P$ of any dimension as
\be
\measuredangle (\mathbbm{n}^\perp,P) \equiv \measuredangle (P, \mathbbm{n}^\perp) := \frac{\pi}{2}- \measuredangle(\mathbbm{n},P).
\ee

Finally, we note that since $\measuredangle \left(v, \mathrm{proj}_Pv\right)= \min_{w \in P} \measuredangle(v, w)$, we have that
\be \label{eq:projection_minimizes_angle}
\measuredangle \left(v, \mathrm{proj}_Pv\right) \leq \measuredangle \left(v, u\right)
\ee
 for every vector $u \in P$.

\end{itemize}
\section{The geodesic walk}\label{sec:Walk}

 The geodesic walk $\{ X_{i} \}_{i \in \mathbb{N}}$ on a manifold (Figure \ref{fig:Walk}), with a fixed geodesic step size T is defined precisely in Algorithm \ref{alg:Geodesic_Walk}:

\begin{algorithm}[H]
\caption{Geodesic Walk \label{alg:Geodesic_Walk}}
\flushleft
\textbf{parameters}: Integration time $T$, manifold $(\mathcal{M},g)$, number of steps $N$.\\
\textbf{input:}  $X_1 \in \mathcal{M}$.\\
\textbf{output:} First $N$ steps $\{X_{i}\}_{i=1}^{N}$ of the geodesic walk Markov chain on $\mathcal{M}$.
\begin{algorithmic}[1]
\For{$i = 1, 2, \ldots, N-1 $}
\\ Sample $U_i \in \mathcal{T}_{X_{i}}$ from the uniform distribution $\mathrm{Unif}(\mathbb{S}(\mathcal{T}_{X_i}))$ on the unit sphere in $\mathcal{T}_{X_i}$.
\\ Set $X_{i+1} = \gamma_{(X_i,U_i)}(T)$.
\EndFor
\end{algorithmic}
\end{algorithm}

That is, the point $X_{i+1}$ is generated from $X_{i}$ by running a geodesic trajectory with initial conditions $(X_i,U_i)$ for a fixed time $T$.

We define the transition kernel $\mathbb{K}$ of the geodesic walk Markov chain $\{X_{i}\}_{i \in \mathbb{N}}$ by
\begin{equation} \label{DefTransKernel}
\mathbb{K}(Q,R) := \mathbb{P}[X_{i+1} \in R | X_i \sim \mathrm{Unif}(Q)]
\end{equation}
for all measurable subsets $Q,R \subset \mathcal{M}$.

For the geodesic walk $x = X_1, X_2, \ldots$, we define $\Pi_{i}(\cdot) = \mathbb{K}^{i}(x,\cdot)$ to be the distribution of $X_i$, and $\Pi = \lim_{i \rightarrow \infty} \Pi_{i}$ to be the stationary distribution of the geodesic walk when it exists.  In Section \ref{sec:Stationary} we will prove that the uniform measure on $\mathcal{M}$ is a stationary distribution of the geodesic walk (later, we show that this is the only stationary distribution).

We note that, in every step of Algorithm \ref{alg:Geodesic_Walk}, a random variable $U_{i}$ was constructed. It is straightforward to see that the sequence $\{ (X_{i}, U_{i}) \}_{i \in \mathbb{N}}$ is a Markov chain on $\mathcal{M}^{\circ}$, which we will call the \textit{phase-space} Markov chain. 

\begin{rem}
Although we write down a fixed integration time $T$, the algorithm still has the correct stationary distribution if i.i.d. random integration times $T_{1},T_{2},\ldots$ are sampled at each step (the proof of Theorem \ref{ThmStatExist} works essentially as written). In practice it may be difficult to choose a specific integration time, and we suspect that tricks from the literature on Hamiltonian Monte Carlo can be adapted to this algorithm (see \textit{e.g.} \cite{hoffman2014no}).
\end{rem}

In referring to the geodesic walk as an ``algorithm", we will assume that we have an oracle for computing geodesic trajectories $\gamma$ with perfect accuracy; we refer to this oracle as an ``idealized geodesic integrator." We drop this assumption in Section \ref{sec:approximating}, where we discuss a computational implementation of the geodesic walk.

\begin{figure}[t]
\begin{center}
\includegraphics[scale=0.4]{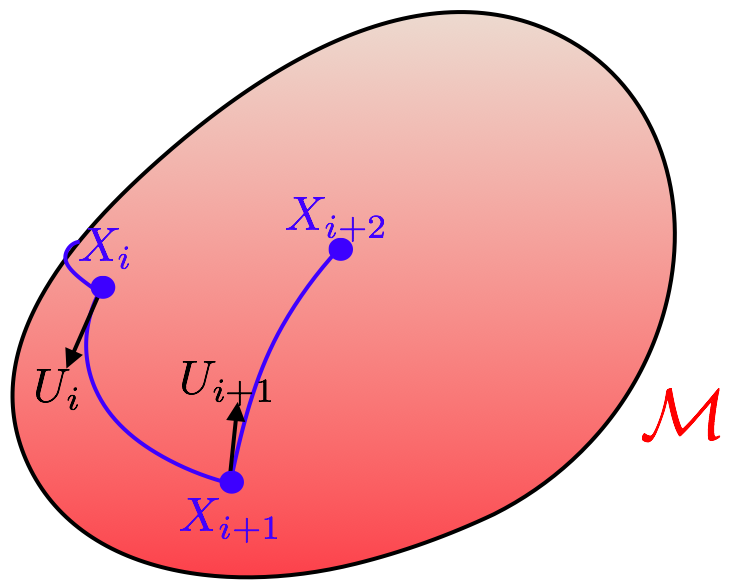}
\end{center}
\caption{The geodesic walk $X_1, X_2, \ldots$ (blue).  Given the current point $X_i$, the next point is generated by independently sampling a uniform random velocity $U_i$ from the unit sphere on the tangent space at $X_i$, and running a geodesic trajectory with initial conditions $(X_i,U_i)$ for a fixed time $T$. \label{fig:Walk}}
\end{figure}

\section{The stationary distribution of the geodesic walk}\label{sec:Stationary}

The aim of this section is to prove that the transition kernel $\mathbb{K}$ defined in Equation \ref{DefTransKernel} has uniform stationary distribution.  To do so we  will use the fact that the Liouville measure is invariant under geodesic flow \cite{Liouville_Invariant3, Liouville_Invariant1, Liouville_Invariant2}. 

Recall the definition of the phase space $\mathcal{M}^{\circ}$ in Equation \ref{EqPhaseSpaceDef}. The Liouville measure $\mathcal{L}$ is defined to be the measure on $\mathcal{M}^{\circ}$ with density given by the product of the volume form $\mathrm{d}\lambda(x)$ on the manifold and the volume form $\mathrm{d}\mu(v)$ on the unit sphere in the tangent space of the manifold at the point $x$ (i.e., the unit-speed velocities):
\begin{equation}
\mathcal{L}(\Omega):= \int_\Omega \mathrm{d}\mu(v) \mathrm{d}\lambda(x).
\end{equation}

Consider a geodesic trajectory on $\mathcal{M}$ with initial position $x$ and initial velocity $v$.  We define $\gamma_{(x,v)}(t)$ and $
\varphi_{(x,v)}(t) := \gamma_{(x,v)}'(t) \in \mathcal{T}_{\gamma_{(x,v)}(t)}$ to be, respectively, the position and velocity of this trajectory at time $t \in \mathbb{R}^{+}$.  Finally, we define
\[\psi_{(x,v)}(t) := (\gamma_{(x,v)}(t), \varphi_{(x,v)}(t))\]
to be the location of our geodesic trajectory in the phase space at time $t \in \mathbb{R}^{+}$.

For notational convenience, we define the map
\begin{equation}
\Psi_t(A) :=  \cup_{(x,v) \in A} \psi_{(x,v)}(t)
\end{equation}
for all $A \subset \mathcal{M}^\circ$ and $t>0$.

We now prove that the volume measure on any closed manifold $\mathcal{M}$ is a stationary measure of the geodesic walk (Theorem \ref{ThmStatExist}); this holds even if $\mathcal{M}$ does not have positive curvature.  In Theorem \ref{thm:uniquenessStat} of Section \ref{sec:contraction} we prove that the stationary distribution is unique under the additional assumption that $\mathcal{M}$ has bounded positive curvature.

\begin{thm}\label{ThmStatExist}
  The uniform distribution on $\mathcal{M}$ is a stationary distribution of the geodesic walk Markov chain defined in Algorithm \ref{alg:Geodesic_Walk}.
\end{thm}

\begin{proof}

We begin by defining two transition kernels $\mathbb{K}^{\circ}_{1}$ and $\mathbb{K}^{\circ}_{2}$ on $\mathcal{M}^{\circ}$. We define the Markov chain $\{(X_{i}^{(1)}, v_{i}^{(1)}) \}_{i \in \mathbb{N}}$ on $\mathcal{M}^{\circ}$ by the deterministic update rule
\begin{equation*}
(X_{i+1}^{(1)}, v_{i+1}^{(1)}) = \psi_{(X_{i}^{(1)},v_{i}^{(1)})}(T)
\end{equation*} 
and define $\mathbb{K}^{\circ}_{1}$ to be the associated transition kernel. We then define the Markov chain $\{(X_{i}^{(2)}, v_{i}^{(2)}) \}_{i \in \mathbb{N}}$ on $\mathcal{M}^{\circ}$ by the update rule
\begin{align*} 
X_{i+1}^{(2)} &= X_{i}^{(2)} \\
v_{i+1}^{(2)} &\sim \mathrm{Unif}(\mathbb{S}(\mathcal{T}_{X_{i+1}^{(2)}}))
\end{align*} 
and define $\mathbb{K}^{\circ}_{2}$ to be the associated transition kernel. We observe that $\mathbb{K}_{1}^{\circ}$ is deterministic, that $\mathbb{K}_{2}^{\circ}$ only ever updates its second coordinate, and finally that
\begin{equation} \label{EqKDecomp}
\mathbb{K}^{\circ} = \mathbb{K}_{2}^{\circ} \, \mathbb{K}_{1}^{\circ}. 
\end{equation}

Since the Liouville measure is invariant under geodesic flow \cite{Liouville_Invariant1, Liouville_Invariant2,Liouville_Invariant3}, we have
\begin{equation}\label{eq:s3}
\mathcal{L}(A) = \mathcal{L}(\Psi_T(A))
\end{equation}
for every measurable $A \subset \mathcal{M}^\circ$. This implies that the Liouville measure $\mathcal{L}$ is invariant under $\mathbb{K}^{\circ}_{1}$:
\begin{equation} \label{EqK1Invariant}
\mathcal{L} \, \mathbb{K}_{1}^{\circ} = \mathcal{L}.
\end{equation}

It is clear by inspection that 
\begin{equation} \label{EqK2Invariant}
\mathcal{L} \, \mathbb{K}_{2}^{\circ} = \mathcal{L}.
\end{equation}

Thus, by Equations \ref{EqKDecomp}, \ref{EqK1Invariant} and \ref{EqK2Invariant},
\begin{equation}
\mathcal{L} \, \mathbb{K}^\circ = \mathcal{L}.
\end{equation}

Let $\{(X_{i},U_{i})\}_{i \geq 0}$ be a Markov chain started at $(X_{1},U_{1}) \sim \mathcal{L}$ and evolving according to $\mathbb{K}^{\circ}$. By the definition of $\mathbb{K}^{\circ}$, the marginal process $\{X_{i}\}_{i \in \mathbb{N}}$ is a Markov chain evolving according to the transition kernel $\mathbb{K}$. Thus, the marginal distribution of $\mathcal{L}$ on its first coordinate $\mathcal{M}$ must be a stationary distribution for $\mathbb{K}$. But this marginal distribution is exactly $\mathrm{Unif}(\mathcal{M})$, completing the proof.
\end{proof}

\section{Coupling the Geodesic Walk}\label{sec:coupling}

We define a coupling of two copies of the geodesic walk: \\

\begin{defn}[Coupling of Geodesic Walk]\label{DefCouplingConstruction}
Fix $x,y \in \mathcal{M}$ and $0 < T < \infty$. We define a pair of stochastic processes $\{(X_{i},Y_{i})\}_{i \in \mathbb{N}}$ with $X_{1} = x$, $Y_{1} = y$ as follows:

Let $X_{1} = x$ and $Y_{1} = y$. For $i \in \mathbb{N}$, inductively sample
\begin{equation*} 
U_{i} \sim \mathrm{Unif}(\mathbb{S}(\mathcal{T}_{X_{i}}))
\end{equation*} 
and set 
\begin{align}
V_{i} &= \phi(U_{i};X_{i}, Y_{i})  \\
X_{i+1} &= \gamma_{(X_{i}, U_{i})}(T) \\
Y_{i+1} &= \gamma_{(Y_{i}, V_{i})}(T).
\end{align}

\end{defn}

This coupling is illustrated in Figure \ref{fig:coupling}.

\begin{figure}[t]
\begin{center}
\includegraphics[scale=0.4]{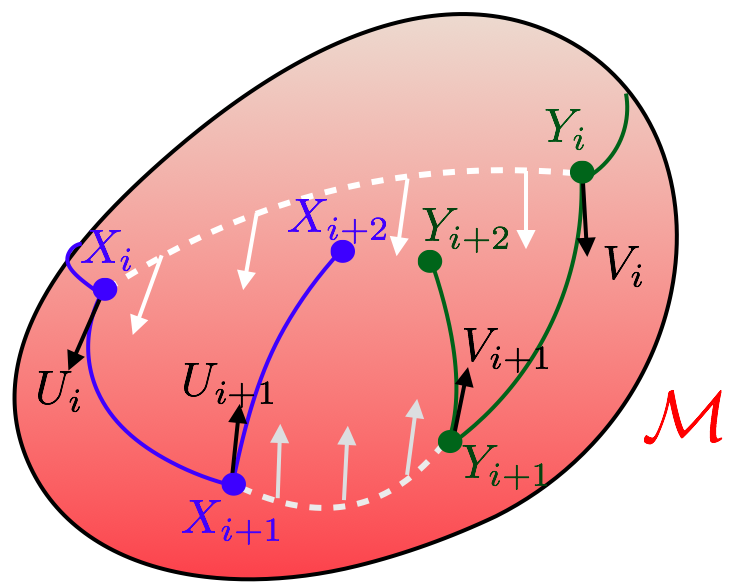}
\end{center}
\caption{Coupling of the two geodesic walks $X_1,X_2,\ldots$ (blue) and $Y_1,Y_2,\ldots$ (green).  The coupling is achieved by setting the initial trajectory velocity $V_i$ at $Y_i$ (black) to be the parallel transport of the initial trajectory velocity $U_i$ at $X_i$ along a minimum geodesic from $X_i$ to $Y_i$ (gray/white dashed line).  Since $\mathcal{M}$ has positive curvature, the distance between the Markov chains contracts rapidly at each step. }\label{fig:coupling}
\end{figure}

This stochastic process is a valid coupling of two copies of the geodesic walk:

\begin{thm}\label{thm:coupling}
Let $\{(X_{i},Y_{i})\}_{i \in \mathbb{N}}$ be as in Definition \ref{DefCouplingConstruction}. Then the marginal processes $\{X_{i}\}_{i \in \mathbb{N}}$ and $\{Y_{i}\}_{i \in \mathbb{N}}$ are each Markov chains with transition kernel $\mathbb{K}$.

\end{thm}
\begin{proof}
It is straightforward to check that $\{X_{i}\}_{i \in \mathbb{N}}$ is exactly the Markov chain defined in Algorithm \ref{alg:Geodesic_Walk} (even the notation is the same). Thus, it remains only to check that $\{ Y_{i} \}_{i \in \mathbb{N}}$ has the correct distribution.

To see that $\{ Y_{i} \}_{i \in \mathbb{N}}$ has the correct distribution, looking at Algorithm \ref{alg:Geodesic_Walk} it is enough to check that (conditional on $\{ Y_{j} \}_{j \leq i}$), $V_{i}$ has uniform distribution on $\mathbb{S}(\mathcal{T}_{Y_{i}})$. 

By the fundamental theorem of Riemannian geometry, parallel transport using the Levi-Civita connection is an affine transformation.  Hence, it preserves the angles between any two vectors $u$ and $u'$ parallel transported to $v=\phi(u;X_i,Y_i)$ and $v' = \phi(u';X_i,Y_i)$, respectively:
\begin{equation}
\measuredangle(u, u') = \measuredangle(v, v').
\end{equation}
Therefore, denoting the conditional density of $U_i$ conditioned on $X_i$ by $f_{U_i}(u| X_i = x)$, and the conditional density of $V_i$ conditioned on $Y_i$ by $f_{V_i}(v| Y_i = y)$, we have
\begin{equation}
f_{V_i}(v|Y_i = y) = f_{U_i}(u| X_i = x) = \frac{1}{\mathrm{Vol}(\mathbb{S}^{d-1})}
\end{equation}
whenever $u$ is a vector on the unit sphere in the tangent plane of $x$, and $v = \phi(u;x,y)$ is its parallel transport to $Y_i$.  Since the Levi-Civita connection also preserves magnitude, the fact that $u$ is on the unit sphere in the tangent plane of $x$ implies that $v$ is on the unit sphere in the tangent plane of $y$.

Since $U_i$ is uniformly distributed on the unit sphere in $\mathcal{T}_{X_{i}}$, this implies that $V_i$ is uniformly distributed on the unit sphere in $\mathcal{T}_{Y_{i}}$.  Therefore the transition kernel of the Markov chain $\{Y_{i}\}_{i \in \mathbb{N}}$ must be $\mathbb{K}$ as well.
\end{proof}

\section{Contraction of Coupled Geodesics}\label{sec:contraction}

In this section we prove a contraction bound on the coupled geodesic walks $X_1,X_2,\ldots$ and $Y_1,Y_2,\ldots$ using the following extension of the Rauch comparison theorem from differential geometry:

\begin{lemma}\label{thm:Global_Rauch}
Let $x,y \in \mathcal{M}$, let $v_{x} \in \mathcal{T}_{x}$ and let $v_{y} = \phi (v_{x}; x,y)$. Then for all $0 \leq T \leq \frac{\uppi}{2 \sqrt{\mathfrak{M}_{2}}}$, we have 

\begin{equation}\label{eq:distance}
\mathrm{dist}(\gamma_{(x,v_{x})}(T), \gamma_{(y,v_{y})}(T)) \leq \cos(\sqrt{\mathfrak{m}_2} T) \, \times \, \mathrm{dist}(x,y).
\end{equation}
\end{lemma}

\begin{rem}
A lower bound for the geodesic distance is proved in \cite{RiemannianGeometryNotes}.  The authors mention that it is possible to prove an upper bound using a similar construction, although they do not do so explicitly.  For this reason, we prove Lemma \ref{thm:Global_Rauch} explicitly here.
\end{rem}

\begin{proof}
Consider the family of geodesics $\{\gamma_{(\omega(\tau;x,y),\phi(\tau;v_x;x,y))}(t) : \tau \in [0,\mathrm{dist}(x,y)]\}$.  Let $J_\tau(t)$ be the Jacobi field associated with this family.

Since $\mathcal{M}$ has curvature bounded above by $\mathfrak{M}_2>0$, the Rauch comparison theorem (comparing $\mathcal{M}$ to a sphere of radius $\frac{1}{\sqrt{\mathfrak{M}_2}}$) \cite{rauch1951contribution, eschenburg1994comparison} gives
\begin{equation}
\|J_\tau(t)\| \geq \cos(\sqrt{\mathfrak{M}_2}t)
\end{equation}
for all 
\begin{equation*}
t \leq \inf \{ t > 0 \, : \, \cos(\sqrt{\mathfrak{M}_{2}}t) = 0 \} = \frac{\uppi}{2 \sqrt{\mathfrak{M}_{2}}}.
\end{equation*}

Thus,
\begin{equation} \label{eq:conj}
\|J_\tau(t)\| >0
\end{equation}
for all $0 \leq t < T < \frac{\uppi}{2\sqrt{\mathfrak{M}_2}}$.
Therefore, since $\|J_\tau(t)\|$ has no zeros on $[0,T)$, applying the Rauch comparison theorem a second time gives (this time comparing $\mathcal{M}$ to a sphere of radius $\frac{1}{\sqrt{\mathfrak{m}_2}}$):
\begin{equation}\label{eq:Rauch}
\|J_\tau(t)\| \leq \cos(\sqrt{\mathfrak{m}_2}t)
\end{equation}
for all $0\leq t < T$.

Define
\begin{equation}
F(\tau) := \mathrm{length}(\{\gamma_{(\omega(s;x,y),\phi(s;v_x;x,y))}(T) : s \in [0,\tau]\})
\end{equation}
for $0 < \tau < \mathrm{dist}(x,y)$. Then
\begin{equation}\label{eq:derivative}
\frac{\mathrm{d}F}{\mathrm{d}\tau} = \|J_\tau(T)\|.
\end{equation}

Therefore, by the fundamental theorem of calculus,
\be
\mathrm{dist}(\gamma_{(x,v_{x})}(T), \gamma_{(y,v_{y})}(T)) &= F(\mathrm{dist}(x,y)) \\
&= \int_0^{\mathrm{dist}(x,y)} \frac{\mathrm{d}F}{\mathrm{d}\tau} \mathrm{d}\tau\\
&\stackrel{{\scriptsize \textrm{Eq. }} \eqref{eq:derivative}}{=} \int_0^{\mathrm{dist}(x,y)} \|J_\tau(T)\| \mathrm{d}\tau\\
&\stackrel{{\scriptsize \textrm{Eq. }} \eqref{eq:Rauch}}{\leq}  \int_0^{\mathrm{dist}(x,y)} \cos(\sqrt{\mathfrak{m}_2} T) d\tau\\
&= \mathrm{dist}(x,y) \times \cos(\sqrt{\mathfrak{m}_2} T).
\ee
\end{proof}

\begin{thm}\label{thm:contraction}
Fix $T\leq \frac{\uppi}{2\sqrt{\mathfrak{M}_2}}$ and let $\{(X_{i},Y_{i})\}_{i \geq 0}$ be coupled as in Definition \ref{DefCouplingConstruction}. Then
\begin{equation}
\mathrm{dist}(X_{i+1}, Y_{i+1}) \leq \mathrm{dist}(X_i, Y_i) \times \cos (\sqrt{\mathfrak{m}_2} T).
\end{equation}
\end{thm}

\begin{proof}
This follows immediately from an application of Lemma \ref{thm:Global_Rauch}.
\end{proof}

\begin{thm} \label{thm:uniquenessStat}
The uniform distribution on $\mathcal{M}$ is the unique stationary distribution of the geodesic walk Markov chain defined in Algorithm \ref{alg:Geodesic_Walk}.
\end{thm}

\begin{proof}
By Theorem \ref{ThmStatExist}, the uniform distribution $\Pi$ on $\mathcal{M}$ is a stationary distribution of $\mathbb{K}$. Assume that there is a second stationary distribution $\Pi' \neq \Pi$. Since $\mathcal{M}$ has finite diameter, this implies $0 < W_{d}(\Pi, \Pi') < \infty$. By Theorem \ref{thm:contraction}, however,
\begin{equation*} 
W_{d}(\Pi, \Pi') = W_{d}(\Pi \, \mathbb{K}, \Pi' \, \mathbb{K}) \leq   \cos (\sqrt{\mathfrak{m}_2} T) \times W_{d}(\Pi, \Pi') < \, W_{d}(\Pi, \Pi'),
\end{equation*}
 since $0 < W_{d}(\Pi, \Pi') < \infty$. This is a contradiction, so no such distribution $\Pi'$ exists.
\end{proof}

\section{Bounding the Mixing Time} \label{sec:mixing}

In this section we bound the mixing time (in Wasserstein Transportation distance) using the approach of \cite{ollivier2009ricci}.  

\begin{thm}\label{thm:mixing}
The Wasserstein mixing profile $\mathfrak{t}_\mathrm{mix}$ of the Markov chain with transition kernel $\mathbb{K}$ defined in Equation \eqref{DefTransKernel} of Section \ref{sec:Walk} and parameter $0 \leq T \leq \frac{\pi}{2 \sqrt{\mathfrak{M}_{2}}}$ satisfies
\begin{equation}
\mathfrak{t}_\mathrm{mix}(\epsilon) \leq \Big\lceil \frac{\log(\epsilon D^{-1})}{\log(\cos (\sqrt{\mathfrak{m}_2} T))} \Big\rceil \ = \mathcal{O}\Bigg(\frac{1}{\mathfrak{m}_2T^2} \cdot \log(D \epsilon^{-1})\Bigg).
\end{equation}
In particular, for $T = \frac{\uppi}{2\sqrt{\mathfrak{M}_2}}$, the mixing time is bounded by
\begin{equation}
\mathfrak{t}_\mathrm{mix}(\epsilon) \leq \Big\lceil \frac{\log(\epsilon D^{-1})}{\log(\cos(\frac{\uppi}{2} \frac{\sqrt{\mathfrak{m}_2}}{\sqrt{\mathfrak{M}_2}}))} \Big\rceil  = \mathcal{O}\Bigg(\frac{\mathfrak{M}_2}{\mathfrak{m}_2} \cdot \log(D \epsilon^{-1})\Bigg).
\end{equation}
\end{thm}

\begin{proof}
Fix $0 \leq T \leq \frac{\uppi}{2\sqrt{\mathfrak{M}_2}}$. Let $\{(X_{i},Y_{i})\}_{i \in \mathbb{N}}$ be two copies of the Markov chain with kernel $\mathbb{K}$, coupled as in Definition \ref{DefCouplingConstruction}. Recall that,  by Theorem \ref{thm:coupling}, this is a valid coupling.
By Theorem \ref{thm:contraction}, this chain satisfies the following contraction inequality:
\begin{equation}
\mathrm{dist}(X_{i+1}, Y_{i+1}) \leq \mathrm{dist}(X_i, Y_i) \times \cos (\sqrt{\mathfrak{m}_2} T ) \quad \quad \forall i>0.
\end{equation}

Applying this contraction inequality repeatedly, we find:
\begin{equation}
\mathrm{dist}(X_{j}, Y_{j}) \leq \mathrm{dist}(X_0, Y_0) \times \cos (\sqrt{\mathfrak{m}_2} T)^j \leq D \times \cos (\sqrt{\mathfrak{m}_2} T)^j \quad \quad \forall j>0.
\end{equation}

This bound immediately implies that, for all $0 < \epsilon < 1$,
\begin{equation}
\mathfrak{t}_\mathrm{mix}(\epsilon) \leq \Big\lceil \frac{\log(\epsilon D^{-1})}{\log(\cos (\sqrt{\mathfrak{m}_2} T))} \Big\rceil.
\end{equation}

In particular, for $T = \frac{\uppi}{2\sqrt{\mathfrak{M}_2}}$, we get:
\begin{equation}
\mathfrak{t}_\mathrm{mix}(\epsilon) \leq  \Big\lceil \frac{\log(\epsilon D^{-1})}{\log(\cos(\frac{\uppi}{2} \frac{\sqrt{\mathfrak{m}_2}}{\sqrt{\mathfrak{M}_2}}))} \Big\rceil.
\end{equation}
\end{proof}

\section{Approximating Geodesics}\label{sec:approximating}

In this section we describe and analyze an approximation to the geodesic walk, showing that it can be used to approximately sample from the uniform distribution on $\mathcal{M}$ to arbitrary accuracy using a dimension-independent number of computations. 

We begin in Section \ref{sec:Approx_Manifold} by assuming that we have access to an oracle $\bigstar$, described at the beginning of Section \ref{sec:Approx_Manifold}, that can approximate short geodesic paths with small error.  We use this oracle to construct an ``approximate" version of the geodesic walk, given in Algorithm \ref{alg:Approx_Manifold}. Finally, we prove that, for any fixed error rate $\epsilon > 0$, this approximate geodesic walk can give samples that are within $\epsilon$ of the uniform distribution on $\mathcal{M}$ in Wasserstein distance using a dimension-independent number of oracle calls.  In Section \ref{sec:Approx_Convex}, we construct $\bigstar$ explicitly for the special case where $\mathcal{M}$ is the boundary of a convex body, using only the basic convex body oracles used in \cite{billiards} (Figure \ref{fig:approx}).

\subsection{Approximating Geodesics on General Positive-Curvature Manifolds} \label{sec:Approx_Manifold}

In Algorithm \ref{alg:Approx_Manifold} we compute recursively an approximation $\gamma^\dagger_{\theta;(x,v)}(T)$ of $\gamma_{(x,v)}(T)$ by repeatedly calling an oracle $\bigstar: \mathcal{M}^\circ \times \mathbb{R}^+ \rightarrow \mathcal{M}^\circ \times \mathbb{R}^+ $.  We use the following notation for the oracle $\bigstar$:

\begin{defn}
An oracle $\bigstar: \mathcal{M}^\circ \times \mathbb{R}^+ \rightarrow \mathcal{M}^\circ \times \mathbb{R}^+ $ is a function that has inputs $(x,v) \in \mathcal{M}^\circ$, and $\theta>0$, and outputs $\psi^\star_{\theta;(x,v)} = \bigg(\gamma^\star_{\theta;(x,v)}, \varphi^\star_{\theta;(x,v)}\bigg) \in \mathcal{M}^\circ$ and $\Delta^\star_{\theta; (x,v)} \in \mathbb{R}^{+}$.
\end{defn}

The oracle tries to approximate the geodesic trajectory $\psi_{(x,v)}(t)$ at $t=\Delta^\star_{\theta; (x,v)}$, where the step size $\Delta^\star_{\theta; (x,v)}$ is determined by the oracle.  The parameter $\theta >0$ allows the user to adjust the accuracy of the oracle's approximation.

For every oracle $\bigstar$ and every $\theta>0$, we define
\begin{equation}\Delta^\star_{\theta;\mathrm{min}} := \inf_{(x,v) \in \mathcal{M}^\circ} \Delta^\star_{\theta; (x,v)}. \end{equation}

We make the following assumptions about the oracle $\bigstar$:

\begin{assumption}\label{assumption:polynomial}
There exist multivariate polynomials $P$ and $Q$ with the property that, at every $(x,v) \in \mathcal{M}^\circ$ and every $0 \leq \theta < \frac{1}{70} \frac{\mathfrak{m}_2}{\mathfrak{M}_2}$, we have
\begin{equation}\label{eq:polynomial}
\frac{1}{\sqrt{\mathfrak{M}_2}}\cdot P\bigg(\theta,\frac{\sqrt{\mathfrak{M}_2}}{\sqrt{\mathfrak{m}_2}}, \frac{\sqrt{\mathfrak{m}_2}}{\sqrt{\mathfrak{M}_2}}\bigg)  \leq \Delta^\star_{\theta; (x,v)} \leq \frac{1}{\sqrt{\mathfrak{M}_2}} \cdot Q\bigg(\theta,\frac{\sqrt{\mathfrak{M}_2}}{\sqrt{\mathfrak{m}_2}}, \frac{\sqrt{\mathfrak{m}_2}}{\sqrt{\mathfrak{M}_2}}\bigg).
\end{equation}
Also, for any fixed $0 < a,b < \infty$, $P(\theta,a,b)$ and $Q(\theta,a,b)$ are strictly increasing in $\theta$ over the interval $0 < \theta < \frac{1}{70} \frac{\mathfrak{m}_2}{\mathfrak{M}_2}$, and $P(0,a,b) = Q(0,a,b) = 0$.
\end{assumption}

We also assume:

\begin{assumption}\label{assumption:accuracy}
There exist constants $\alpha >0$, $\beta >0$, and ${k}\in \mathbb{N}$ such that
\begin{equation} \label{IneqAlphaAssumption}
\mathrm{dist}(\gamma_{(x,v)}(\Delta^\star_{\theta;{(x,v)}}), \gamma^\star_{\theta; (x,v)}) \leq  \alpha \cdot  \theta^{k} \cdot \Delta^\star_{\theta; (x,v)}
\end{equation}
and also
\begin{equation} \label{IneqBetaAssumption}
\|\varphi^\star_{\theta; (x,v)} - \overline{\varphi_{(x,v)}(\Delta^\star_{\theta;{(x,v)}})}\| \leq \beta \cdot \sqrt{\mathfrak{M}_2} \cdot \theta^{k} \cdot \Delta^\star_{\theta;{(x,v)}}
\end{equation}
for every $(x,v) \in \mathcal{M}^\circ$ and $0 \leq \theta < \frac{1}{70} \frac{\mathfrak{m}_2}{\mathfrak{M}_2}$,\\ where $\overline{\varphi_{(x,v)}(\Delta^\star_{\theta;{(x,v)}})} := \phi(\varphi_{(x,v)}(\Delta^\star_{\theta;{(x,v)}}); \gamma_{(x,v)}(\Delta^\star_{\theta;{(x,v)}}), \gamma_{\theta;(x,v)}^\star).$
\end{assumption}

For the rest of Section \ref{sec:Approx_Manifold}, we will assume that the oracle $\bigstar$ satisfies Assumptions \ref{assumption:polynomial} and \ref{assumption:accuracy}.

\begin{algorithm}[H]
\caption{Numerical Approximation of Geodesic Trajectory \label{alg:Approx_Manifold}}
\flushleft
\textbf{parameters:} $\theta>0$, $T>0$.\\
 \textbf{input:}  Manifold $(\mathcal{M},g)$ with sectional curvature bounded above by $\mathfrak{M}_2$.\\
 \textbf{input:}   Initial point $x_0 \in \mathcal{M}$, initial velocity $v_0 \in \mathcal{T}_{x_0}\mathcal{M}$.\\
 \textbf{input:}   Oracle $\bigstar$ returning $\psi^\star_{\theta; (x,v)}$ and $\Delta^\star_{\theta; (x,v)}$ for any $(x,v) \in \mathcal{M}^\circ$.\\
 \textbf{output:} $\gamma^\dagger_{\theta; (x_0,v_0)}(T)$.
\begin{algorithmic}[1]
\State Set $i=0$, $\Delta_{0} = 0$. 
\State Set $x_0^\dagger = x_0$ and $v_0^\dagger = v_0$.
\While{$\sum_{j=0}^i \Delta_j \leq T$}
\\ Call the oracle $\bigstar$ to compute $(x_{i+1}^\dagger, v_{i+1}^\dagger) = \psi^\star_{\theta; (x_i^\dagger,v_i^\dagger)}$ and $\Delta_{i+1} = \Delta^\star_{\theta;{(x_i^\dagger,v_i^\dagger)}}$.
\\ Set $i = i+1$.
\EndWhile
\\Set $\delta = T - \sum_{j=1}^{i }\Delta_j$. 
\\ Using the bisection method, iteratively call $\bigstar$ to find a value $\hat{\theta}$ for which $\delta - \frac{\theta^{k}}{\sqrt{\mathfrak{M}_2}} \leq \Delta^\star_{\hat{\theta};{(x_i^\dagger,v_i^\dagger)}} \leq \delta$. Set $\Delta_{i+1} := \Delta^\star_{\hat{\theta};{(x_i^\dagger,v_i^\dagger)}}$.
\\ Call the oracle $\bigstar$ to compute $x_{i+1}^\dagger = \gamma^\star_{\hat{\theta}; (x_i^\dagger,v_i^\dagger)}$.
\\ Set $\gamma^\dagger_{\theta; (x_0,v_0)}(T) = x_{i+1}^{\dagger}$.
\end{algorithmic}
\end{algorithm}

Using the notation of Algorithm \ref{alg:Geodesic_Walk}, the geodesic walk Markov chain is defined recursively by
\begin{equation} X_{i+1}:= \gamma_{(X_i,U_i)}(T) \quad \quad \forall i = 0,1,2,\ldots \end{equation}
where $U_i$ is uniformly distributed on $\mathbb{S}(\mathcal{T}_{X_i})$.
For fixed $\theta > 0$, we can use Algorithm \ref{alg:Approx_Manifold} to recursively generate an approximate geodesic walk $X_1^\theta, X_2^\theta,\ldots$ coupled to the geodesic walk $X_1, X_2, \ldots$, using the following recursion:
\be \label{EqDefGeodesicWalkApprox}
X^\theta_0 &= X_0 \\
X_{i+1}^\theta &= \gamma^\dagger_{\theta; (X^\theta_i,U_{i}')}(T) \quad \quad \forall i = 0,1,2,\ldots, 
\ee

where $U_{i}' = \phi(U_{i}; X_{i}, X_{i}^{\theta})$. For $i \in \mathbb{N}$, we let $\Pi^\theta_i$ be the distribution of $X^\theta_i$.

\begin{figure}[t]
\begin{center}
\begin{tikzpicture}[scale=3.2]

\draw[thick] (-0.985,-0.6)--(-0.92,0.1);
\node at (-0.92,-0.47) {$\theta$};
\node[blue] at (-1.1,-0.6) {$x_0$};

\draw[thick] (-0.92,0.1)--(-0.3,1);
\node at (-0.75,0.26) {$\theta$};
\node at (-1.05,0.12) {$x_1^\dagger$};

\draw[thick] (-0.3,1)--(0.13,1.1);
\node at (-0.21,0.95) {$\hat{\theta}$};
\node at (-0.34,1.13) {$x_2^\dagger$};

\draw [red, thick, domain=-1.118:1.118, samples=100] plot ({sqrt(1-0.8*abs(\x^2))/(1+0.3*\x)},\x);
\draw [red, thick, domain=-1.118:1.118, samples=100] plot ({-sqrt(1-0.8*abs(\x^2))/(1+0.3*\x)+0.04},\x);

\draw [blue, line width=2pt, domain=1.03:1.118, samples=100] plot ({sqrt(1-0.8*abs(\x^2))/(1+0.3*\x)},\x);
\draw [blue, line width=2pt, domain=-0.6:1.118, samples=100] plot ({-sqrt(1-0.8*abs(\x^2))/(1+0.3*\x)+0.04},\x);

\draw[blue,fill=blue] (-0.985,-0.6) circle (.17ex);
\draw[blue,fill=blue] (0.29,1.03) circle (.17ex);
\node[blue] at (0.4,0.9) {$\gamma_{(x_0,v_0)}(T)$};


\draw[black,fill=black] (0.13,1.1) circle (.17ex);
\node at (0.4,1.2) {$x_3^\dagger = \gamma_{\theta;(x_0,v_0)}^\dagger(T)$};


\draw[thick] (-0.967,-0.45) arc (40:120:0.04);
\draw[thick] (-0.79,0.3) arc (20:100:0.04);
\draw[thick] (-0.16,1.03) arc (-20:60:0.037);

\node[red] at (1,-1) {$\mathcal{M}$};

\draw[->, thick, color={rgb:black,1;blue,6}] (-0.985,-0.6)--(-1.1,-0.3);

\node[blue] at (-1.2,-0.35) {$v_0$};

\end{tikzpicture}
\end{center}
\caption{Algorithm for approximating geodesics.  In this example $\mathcal{M}$ is the boundary of a convex body, so one can generate each step in the approximation as a line $\ell_i$ leaving the current approximation point $x_i^\dagger$ on the boundary at a fixed angle $\theta$.  The next approximation point $x_{i+1}^\dagger$ is the next point when $\ell_i$ intersects the boundary.  Lemmas \ref{thm:approx_manifold} and \ref{thm:approx_convex} show that the final point $x_{i_\mathrm{max}+1}^\dagger = \gamma_{\theta;(x_0,v_0)}^\dagger(T)$ (here $i_\mathrm{max}$ = 2) approximates the final point $\gamma_{(x_0,v_0)}(T)$ of the true geodesic trajectory to arbitrary accuracy $\epsilon$ in $i_\mathrm{max}(\epsilon)$ steps, where $i_\mathrm{max}(\epsilon)$ is independent of the dimension, if $\theta$ is appropriately chosen.  \label{fig:approx}}
\end{figure}
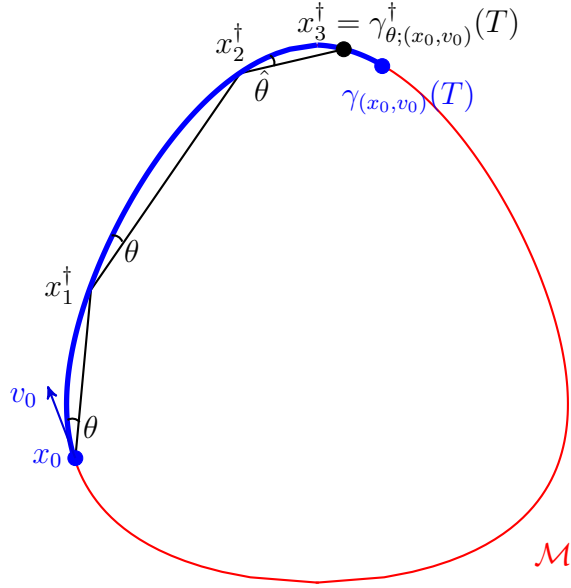

The following lemma shows that our approximate geodesic chain remains close to the actual geodesic chain:

\begin{lemma}\label{thm:approx_manifold}
 For every $(x_0,v_0) \in \mathcal{M}^\circ$, every $0 \leq \theta < \frac{1}{70} \frac{\mathfrak{m}_2}{\mathfrak{M}_2}$, and every $0 \leq T \leq \frac{\uppi}{2 \sqrt{\mathfrak{M}_{2}}}$,  Algorithm \ref{alg:Approx_Manifold} makes at most $\lceil \frac{T}{\Delta_{\theta; \mathrm{min}}^{\star}}\rceil + \mathcal{O}(\log(\frac{\sqrt{\mathfrak{M}_2}}{\sqrt{\mathfrak{m}_2}} \times \frac{\sqrt{\mathfrak{M}_2}}{\theta^{k}}))$ calls of the oracle $\bigstar$ and returns a point $\gamma_{\theta; (x_0,v_0)}^\dagger(T)$ such that
\begin{equation}
\mathrm{dist}(\gamma_{(x_0,v_0)}(T), \gamma_{\theta; (x_0,v_0)}^\dagger(T)) \leq \bigg[1 + \frac{\uppi}{2}\alpha + (\frac{\uppi}{2})^2 \beta\bigg] \cdot \frac{1}{\sqrt{\mathfrak{M}_2}} \cdot \theta^{k}.
\end{equation}
\end{lemma}

\begin{proof}

\textbf{Counting the number of $\bigstar$ Oracle calls:} Define $i_\mathrm{max} := \max \{ i \in \mathbb{N} \, : \, \sum_{j=1}^{i+1} \Delta_j \leq T \}$. The number of $\bigstar$ oracle calls required for Step 4 in Algorithm \ref{alg:Approx_Manifold} is at most $\lceil \frac{T}{\Delta_{\theta; \mathrm{min}}^{\star}} \rceil$, since the largest possible value of $i$ such that $\sum_{j=1}^{i+1} \Delta_j \leq T$ is $\lfloor \frac{T}{\Delta_{\theta; \mathrm{min}}^{\star}}\rfloor$.  Assumption \ref{assumption:polynomial} implies that the number of $\bigstar$ oracle calls the bisection method makes in Step 8 of Algorithm \ref{alg:Approx_Manifold} to find a value $\hat{\theta}$ for which $\Delta_{i+1} := \Delta^\star_{\hat{\theta};{(x_i^\dagger,v_i^\dagger)}}$ satisfies $\delta - \frac{\theta^{k}}{\sqrt{\mathfrak{M}_2}} \leq \Delta_{i+1} \leq \delta$ is at most $\mathcal{O}(\log(\frac{\sqrt{\mathfrak{M}_2}}{\sqrt{\mathfrak{m}_2}} \times \frac{\sqrt{\mathfrak{M}_2}}{\theta^{k}}))$.  Step 9 makes exactly one call of the oracle $\bigstar$. \\

\textbf{Bounding the error:} Fix $i \in \mathbb{N}$ and let $x_i^\dagger$, $v_i^\dagger$, $\Delta_i$, $\theta$, and $\hat{\theta}$ be as in Algorithm \ref{alg:Approx_Manifold}. Define $(x_i, v_i) := \psi_{(x_{i-1}^\dagger,v_{i-1}^\dagger)}(\Delta_i)$, and  $\bar{v_i} := \phi(v_i;x_i, x_i^\dagger).$ By Inequality \ref{IneqBetaAssumption},  we have $\|v_i^\dagger - \bar{v_i}\| \leq \beta \cdot \sqrt{\mathfrak{M}_2} \cdot \theta^{k} \cdot \Delta_i$. Therefore, comparing to a triangle in a sphere of radius $\frac{1}{\sqrt{\mathfrak{M}_2}}$ by the Toponogov triangle comparison theorem (see Chapter 11 of \cite{petersen2006riemannian}), we have 
\begin{equation} \label{ApproxMainTopo1}
\mathrm{dist}(\gamma_{(x_i^\dagger,\bar{v_i})}(t), \gamma_{(x_{i}^{\dagger},v_{i}^{\dagger})}(t))\leq t \cdot \|v_i^\dagger - \bar{v_i}\| \leq t \cdot \beta \cdot \sqrt{\mathfrak{M}_2} \cdot \theta^{k} \cdot \Delta_i
\end{equation}
for all $0\leq t \leq \frac{\uppi}{2 \sqrt{\mathfrak{M}_2}}$.

By Inequality \ref{IneqAlphaAssumption} we have $\mathrm{dist}(x_i, x_i^\dagger) \leq \alpha \cdot \theta^{k} \cdot \Delta_i$.  Therefore, by Lemma \ref{thm:Global_Rauch},
\begin{equation} \label{ApproxMainTopo2}
\mathrm{dist}(\gamma_{(x_i, v_i)}(t), \gamma_{(x_i^\dagger,\bar{v_i})}(t))\leq \mathrm{dist}(x_i, x_i^\dagger) \leq \alpha \cdot \theta^{k} \cdot \Delta_i
\end{equation}
for all $0\leq t \leq \frac{\uppi}{2 \sqrt{\mathfrak{M}_2}}$.

Define the ``remaining time after the $i$'th step" to be $T_i:= \sum_{j=i+1}^{i_\mathrm{max}+1} \Delta_j$. Then
\be
\mathrm{dist}(\gamma_{(x_0,v_0)}(T), \gamma_{\theta;(x_0,v_0)}^\dagger(T)) &\leq \mathrm{dist}(\gamma_{(x_0,v_0)}(T),\gamma_{(x_0,v_0)}(T_{0}) ) \\
& \qquad +  \mathrm{dist}(\gamma_{(x_0,v_0)}(T_0), x^\dagger_{i_\mathrm{max}})\\
&\leq T-T_0 + \mathrm{dist}(\gamma_{(x_0,v_0)}(T_0), x^\dagger_{i_\mathrm{max}})\\
&\leq T - T_0+ \sum_{i=1}^{i_\mathrm{max}+1} (\mathrm{dist}(\gamma_{(x_i, v_i)}(T_i), \gamma_{(x_i^\dagger,\bar{v_i})}(T_i)) \\
& \qquad + \mathrm{dist}(\gamma_{(x_{i}^{\dagger},\bar{v_i})}(T_i), \gamma_{(x_i^\dagger,v_{i}^{\dagger})}(T_i)))\\
&\leq \frac{\theta^{k}}{\sqrt{\mathfrak{M}_2}} + \sum_{i=1}^{i_\mathrm{max}+1} ( \alpha \cdot \theta^{k} \cdot \Delta_i + T_i \cdot \beta \cdot \sqrt{\mathfrak{M}_2} \cdot \theta^{k} \cdot \Delta_i )\\
&\leq \frac{\theta^{k}}{\sqrt{\mathfrak{M}_2}} + [  \alpha \cdot \theta^{k} + T_0 \cdot \beta \cdot \sqrt{\mathfrak{M}_2} \cdot \theta^{k}] \cdot \sum_{i=0}^{i_\mathrm{max}+1}  \Delta_i\\
&\leq \frac{\theta^{k}}{\sqrt{\mathfrak{M}_2}} + [  \alpha \cdot \theta^{k} + T_0 \cdot \beta \cdot \sqrt{\mathfrak{M}_2} \cdot \theta^{k}] \cdot T_0\\
&\leq \frac{\theta^{k}}{\sqrt{\mathfrak{M}_2}} +  [\alpha + \frac{\uppi}{2\sqrt{\mathfrak{M}_2}} \cdot \beta \cdot \sqrt{\mathfrak{M}_2}] \cdot \frac{\uppi}{2\sqrt{\mathfrak{M}_2}} \cdot \theta^{k},\\
&= \bigg[1 + \frac{\uppi}{2}\alpha + (\frac{\uppi}{2})^2 \beta\bigg] \cdot \frac{1}{\sqrt{\mathfrak{M}_2}} \cdot \theta^{k}.
\ee
The second inequality holds because $\gamma_{(x_0,v_0)}(t)$ has unit velocity for every $t\in \mathbb{R}$. The third inequality is the triangle inequality (see Figure \ref{fig:Toponogov}). The bound on $T-T_0$ in the fourth inequality is a consequence of Step 8 of Algorithm \ref{alg:Approx_Manifold}, while the bounds on the other two terms follow from Equations \ref{ApproxMainTopo1} and \ref{ApproxMainTopo2}.  The last inequality uses our assumption that $T_0 \leq T \leq \frac{\uppi}{2 \sqrt{\mathfrak{M}_{2}}}$.
\end{proof}

\begin{figure} 
\begin{center}
\includegraphics[scale=0.32]{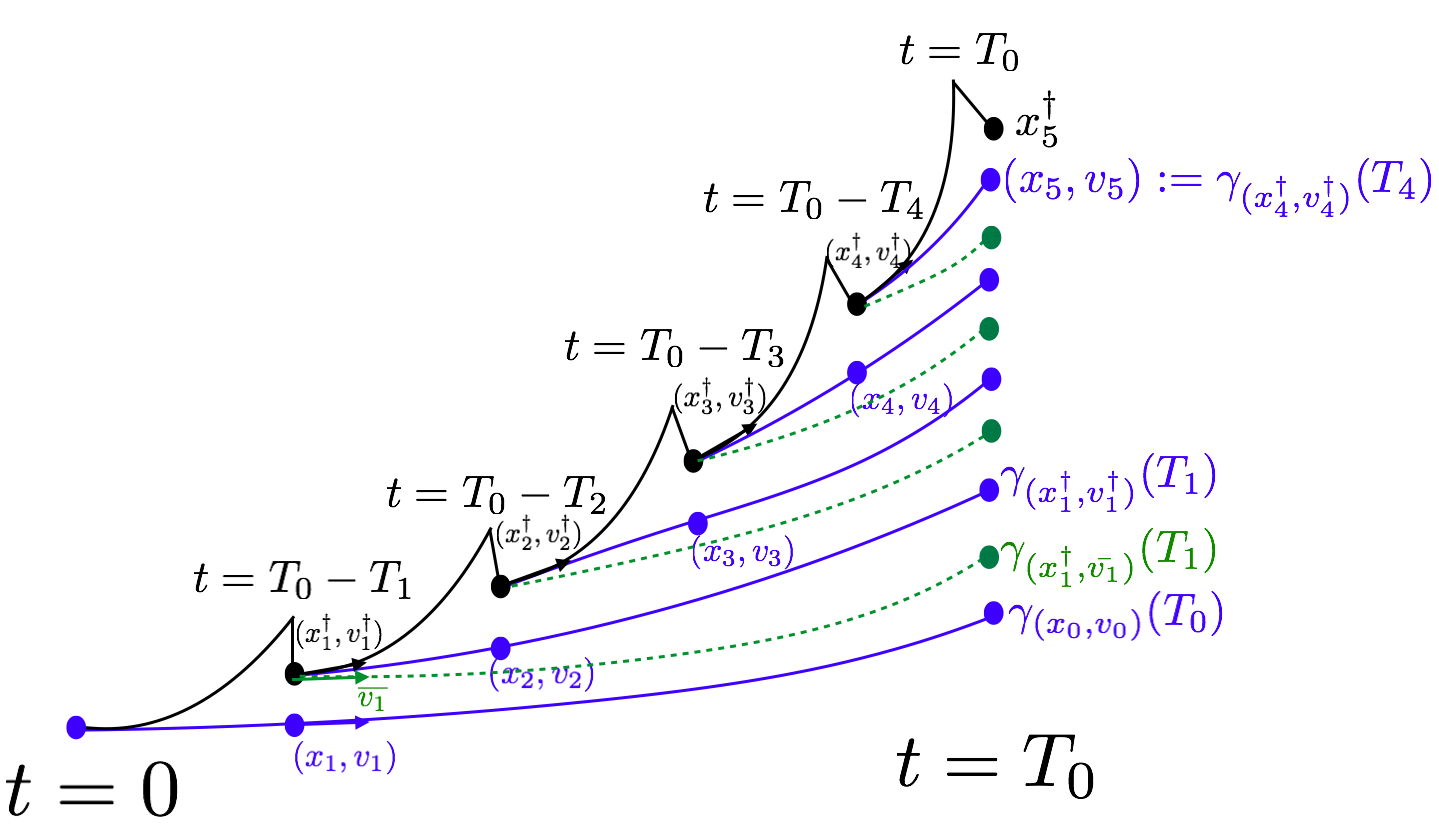}
\end{center}
\caption{This is an illustration of the proof of Lemma \ref{thm:approx_manifold}.  Steps taken by the oracle $\bigstar$ (which is illustrated here as going along a short curved path and then projecting its position $x_i^\dagger := \gamma^\star_{(x_{i-1}^\dagger,v_{i-1}^\dagger)}$ (black dot) and velocity $v_i^\dagger := \gamma^\star_{(x_{i-1}^\dagger,v_{i-1}^\dagger)}$ (black arrow) back onto the phase space) are in black.  The true geodesic paths are blue curves or green dashed curves.  Only the geodesic path $\gamma_{(x_0,v_0)}(t)$ on the bottom belongs to the true geodesic walk.  We imagine the other geodesic paths to help us bound the error.  The distance between the blue dot $x_i$ and black dot $x^\dagger_i$ at each time $t=T_0-T_i$, and the angle between the velocities $\overline{v_i}$ and $v_i^\dagger$, where $v_i$ is the velocity at the blue dot and $\overline{v_i} := \phi(v_i, x_i,x_i^\dagger)$, are bounded because of our assumptions on the accuracy of the oracle $\bigstar$.  The distance from any green dot at $t=T_0$ to the blue dot directly below it is bounded using Lemma \ref{thm:Global_Rauch}.  The distance from that same green dot to the blue dot directly above it is bounded using the Toponogov triangle comparison theorem.}\label{fig:Toponogov}
\end{figure}

We show that the random walk described by Algorithm \ref{alg:Approx_Manifold}  can be used to sample approximately uniformly from $\mathcal{M}$. We first need a simple generic bound:

\begin{lemma} \label{ThmGenericApproxErrorBound}
Let $K$ be a transition kernel on metric space $(\Omega,d)$ with unique stationary measure $\mu$. Assume that there exists some  \textit{contraction coefficient} $\kappa > 0$ so that $K$ satisfies
\begin{equation} \label{SimpAss2}
W_{d}(K(x,\cdot),K(y,\cdot)) \leq (1 - \kappa) d(x,y)
\end{equation}  
for all $x, y \in \Omega$. Let $Q$ be a transition kernel on  $(\Omega,d)$ with stationary measure $\nu$. Assume that there exists some $\delta \geq 0$ so that
\begin{equation} \label{SimpAss1}
\sup_{x \in \Omega} W_{d}(K(x,\cdot), Q(x,\cdot)) < \delta.
\end{equation} 
For fixed $x \in \Omega$ and $Y \sim \mu$, define the \textit{eccentricity} $E(x) = \mathbb{E}[d(x,Y)]$ and assume that $E(x) < \infty$ for all $x \in \Omega$. Then $Q$ satisfies 
\begin{equation} \label{WassContConc1}
W_{d}(Q^{t}(x,\cdot), \mu) \leq (1 - \kappa)^{t} \, E(x) + \frac{\delta}{\kappa}
\end{equation} 
for all $x \in \Omega$ and $t \in \mathbb{N}$. Under the further assumption that $\sup_{x} E(x) < \infty$, we also have 
\begin{equation} \label{WassContConc2}
W_{d}(\mu, \nu) \leq \frac{\delta}{ \kappa}.
\end{equation} 
\end{lemma}

\begin{proof}
Fix a tolerance $\eta > 0$ and starting points $x,y \in \Omega$. We begin by constructing a coupling of $Q(x,\cdot)$ and $K(y,\cdot)$. 

By Inequality \ref{SimpAss2}, it is possible to couple two random variables $X' \sim K(x,\cdot)$ and $Y' \sim K(y,\cdot)$ so that 
\begin{equation} \label{CoupCon1}
\mathbb{E}[d(X',Y')] \leq (1 - \kappa) d(x,y) + \eta.
\end{equation} 
Furthermore, by Inequality \ref{SimpAss1} and the standard gluing lemma (see Chapter 1 of \cite{villani2008optimal}), it is possible to couple $X',Y'$ to the random variable $X \sim Q(x,\cdot)$ so that 
\begin{equation} \label{CoupCon2}
\mathbb{E}[d(X,X')] \leq \delta + \eta.
\end{equation} 
Combining Inequalities \ref{CoupCon1} and \ref{CoupCon2}, it is possible to couple $X \sim Q(x,\cdot)$ and $Y' \sim K(y,\cdot)$ so that 
\be  \label{IneqContM}
\mathbb{E}[d(X,Y')] &\leq \mathbb{E}[d(X,X')] + \mathbb{E}[d(X',Y')] \\
&\leq \delta + (1-\kappa)d(x,y) + 2 \eta.
\ee

We denote by $M$ the kernel on $\Omega^{2}$ given by this coupling of $Q(x,\cdot)$ and $K(y,\cdot)$.

Let $x \in \Omega$, and let $Y \sim \mu$. Let $\{(X_{t},Y_{t})\}_{t \geq 0}$ be a Markov chain evolving according to the kernel $M$ with initial conditions $(X_{0},Y_{0}) = (x,Y)$. By Inequality \eqref{IneqContM}, we have for all $t \in \mathbb{N}$ that 
\begin{align*} 
\mathbb{E}[d(X_{t},Y_{t})] &\leq \mathbb{E}[ (1 - \kappa) \mathbb{E}[d(X_{t-1},Y_{t-1})] + \delta + 2 \eta] \\
&\leq (1-\kappa)^{2} \mathbb{E}[d(X_{t-2},Y_{t-2})] + (\delta + 2 \eta)(1 + (1 - \kappa)) \\
& \leq \ldots \\
&\leq (1 - \kappa)^{t} \mathbb{E}[d(X_{0},Y_{0})] + \frac{\delta + 2 \eta}{\kappa}.
\end{align*} 
Since this holds for arbitrary $\eta > 0$ and $Y_{t} \sim \mu$ for all $t \in \mathbb{N}$, we have 
\begin{equation} 
W_{d}(Q^{t}(x,\cdot), \mu) \leq (1 -\kappa)^{t} \, E(x) + \frac{\delta}{\kappa}.
\end{equation} 
This completes the proof of Inequality \ref{WassContConc1}. Inequality \ref{WassContConc2} follows immediately from letting $t$ go to infinity.
\end{proof}

We apply this bound:

\begin{lemma} \label{IneqNumStepsAppr}
Let $\epsilon >0$.  Set $T = \frac{\uppi}{2 \sqrt{\mathfrak{M}_{2}}}$ and set
\be \label{EqThetaIDef}
 \theta(\epsilon) &= \min \left(\left[ \epsilon \cdot \sqrt{\mathfrak{M}_2} \cdot \frac{1 - \cos(\frac{\uppi \sqrt{\mathfrak{m}_{2}}}{2 \sqrt{\mathfrak{M}_{2}}})}{2 (1 + \frac{\uppi}{2}\alpha + (\frac{\uppi}{2})^2\beta)} \right]^{\frac{1}{k}}, \, \, \, \, \, \, \, \frac{1}{70} \frac{\mathfrak{m}_2}{\mathfrak{M}_2} \right), \\
I(\epsilon) &= \left \lceil \frac{\log(\frac{\epsilon \, \sqrt{\mathfrak{m}_{2}} }{2 \uppi})}{\log(\cos(\frac{\uppi \sqrt{\mathfrak{m}_{2}}}{2 \sqrt{\mathfrak{M}_{2}}}))} \right\rceil.
\ee

Then the distribution $\Pi_{i}^{\theta(\epsilon)}$ of the $i$'th step of the chain described in Equation \ref{EqDefGeodesicWalkApprox} satisfies 
\begin{equation}
W_{d}(\Pi,\Pi^{\theta(\epsilon)}_{i}) \leq \epsilon
\end{equation}
for all $i > I(\epsilon)$. 
\end{lemma}

\begin{proof}

We apply Lemma \ref{ThmGenericApproxErrorBound}, with $K = \mathbb{K}$ and $Q$ the kernel of the Markov chain in Equation \ref{EqDefGeodesicWalkApprox} with parameters $\theta = \theta(\epsilon)$ and $T = \frac{\uppi}{2 \sqrt{\mathfrak{M}_{2}}}$. We keep the notation of Lemma \ref{ThmGenericApproxErrorBound} and check that the assumptions are satisfied.

By Theorem \ref{thm:contraction}, Inequality \ref{SimpAss2} is satisfied with
\begin{equation*}
\kappa = 1 - \cos(\frac{\uppi \, \sqrt{\mathfrak{m}_{2}}}{2 \sqrt{\mathfrak{M}_{2}}}).
\end{equation*}
By Lemma \ref{thm:approx_manifold}, Inequality  \ref{SimpAss1} is satisfied with
\begin{equation*}
\delta = \frac{1 + \frac{\uppi}{2}\alpha + (\frac{\uppi}{2})^2\beta}{\sqrt{\mathfrak{M}_2}} \cdot \theta.
\end{equation*}
Finally, by Inequality \ref{IneqDiameterMyersBonnet}, the diameter of $\mathcal{M}$ satisfies 
\begin{equation*}
D \leq \frac{\pi}{\sqrt{\mathfrak{m}_2}}.
\end{equation*}

Applying Lemma \ref{ThmGenericApproxErrorBound} with these values of $\kappa, \delta$ and $D$ completes the proof of this theorem.

\end{proof}

This immediately gives the following dimension-free bound on the number of oracle calls required to obtain a sample from the uniform distribution on $\mathcal{M}$ with error less than some fixed $\epsilon > 0$:

\begin{thm}\label{thm:e1}
Fix $\epsilon > 0$ and let $\theta(\epsilon), I(\epsilon)$ be as in Equation \ref{EqThetaIDef}. Also let $X_{I(\epsilon) + 1}^{\theta(\epsilon)}$ be the  $(I(\epsilon)+1)$'st point generated by Equation \ref{EqDefGeodesicWalkApprox} with parameters $\theta = \theta(\epsilon)$, $T = \frac{\uppi}{2 \sqrt{\mathfrak{M}_{2}}}$.

Then the distribution $\Pi_{I(\epsilon) + 1}^{\theta(\epsilon)}$ of $X_{I(\epsilon) + 1}^{\theta(\epsilon)}$ satisfies
\begin{equation} \label{IneqPenConc1}
W_{d}(\Pi_{I(\epsilon) + 1}^{\theta(\epsilon)}, \Pi) \leq \epsilon,
\end{equation}

and the number $N(\epsilon)$ of $\bigstar$ oracle calls used to generate $X_{I(\epsilon) + 1}^{\theta(\epsilon)}$ satisfies 
\be \label{IneqPenConc2}
 N(\epsilon) \leq& \frac{\log(\frac{\epsilon \, \sqrt{\mathfrak{m}_{2}} }{2 \uppi})}{\log(\cos(\frac{\uppi \sqrt{\mathfrak{m}_{2}}}{2 \sqrt{\mathfrak{M}_{2}}}))} \times \bigg( \lceil \frac{\uppi}{2 \, \sqrt{\mathfrak{M}_{2}} \, \Delta^\star_{\theta(\epsilon); \mathrm{min}}} \rceil + \mathcal{O}(\log(\frac{\sqrt{\mathfrak{M}_2}}{\sqrt{\mathfrak{m}_2}} \times \frac{\sqrt{\mathfrak{M}_2}}{\theta^{k}(\epsilon)})) \bigg)\\
= &\mathcal{O}^*\bigg(\frac{\mathfrak{M}_2}{\mathfrak{m}_2} \times \frac{1}{\sqrt{\mathfrak{M}_2} \Delta^\star_{\theta(\epsilon); \mathrm{min}}}\bigg).
\ee

\end{thm}

\begin{proof}

Inequality \ref{IneqPenConc1} follows immediately from Lemma \ref{IneqNumStepsAppr}.

By Lemma \ref{thm:approx_manifold}, 
\be
N(\epsilon) &\leq I(\epsilon) \times \bigg(\lceil \frac{T}{\Delta^\star_{\theta(\epsilon); \mathrm{min}}}\rceil + \mathcal{O}(\log(\frac{\sqrt{\mathfrak{M}_2}}{\sqrt{\mathfrak{m}_2}} \times \frac{\sqrt{\mathfrak{M}_2}}{\theta^{k}(\epsilon)}))) \bigg) \\
&\leq \frac{\log(\frac{\epsilon \, \sqrt{\mathfrak{m}_{2}} }{2 \uppi})}{\log(\cos(\frac{\uppi \sqrt{\mathfrak{m}_{2}}}{2 \sqrt{\mathfrak{M}_{2}}}))} \times \bigg( \lceil \frac{\uppi}{2 \, \sqrt{\mathfrak{M}_{2}} \, \Delta^\star_{\theta(\epsilon); \mathrm{min}}} \rceil + \mathcal{O}(\log(\frac{\sqrt{\mathfrak{M}_2}}{\sqrt{\mathfrak{m}_2}}\times \frac{\sqrt{\mathfrak{M}_2}}{\theta^{k}(\epsilon)}))) \bigg).
\ee
This completes the proof of the Theorem.

\end{proof}

\subsection{Approximating Geodesics on Convex body Boundaries}\label{sec:Approx_Convex}

In this section we show how to construct an oracle that satisfies the requirements of Lemma \ref{thm:approx_manifold} in the important special case when $\mathcal{M} = \partial \mathcal{K}$ is the boundary of a convex body $\mathcal{K}$.  To build our oracle $\bigstar$ (Algorithm \ref{alg:Approx_Convex}), we use only the same basic inclusion-exclusion oracle and oracle for the normal line to $\partial \mathcal{K}$ used in the stochastic billiards algorithm of \cite{billiards}.  In particular, one can construct an intersection oracle that returns the intersection of any line with $\partial \mathcal{K}$ using the bisection method and the inclusion-exclusion oracle, which we will use to build $\bigstar$.

 In the remainder of the paper, for any line segment $\ell_{a,b} \subset \mathbb{R}^{d+1}$ with endpoints $a,b \in \mathbb{R}^{d+1}$, we denote its Euclidean length by $\mathrm{length}(\ell_{a,b}) := \|b-a\|$.  For every $y\in \mathcal{M} \equiv \partial \mathcal{K}$, define $\mathbbm{n}(y)$ to be the unit normal vector to $\mathcal{M}$ at $y$ pointing into $\mathcal{K}$, and denote by ``$\mathrm{D}_{\mathsf{u}} w(y)$" the directional derivative in the direction ${\mathsf{u}}$ of a vector-valued function $w: \mathcal{M} \mapsto \mathbb{R}^{d+1}$.  We will denote by $\mathbbm{n}^{\perp}(\cdot)$ the unique codimension-$1$ subspace orthogonal to $\mathbbm{n}(\cdot)$.

\begin{algorithm}[H]
\caption{An oracle $\bigstar$ in the special case where $\mathcal{M}$ is the boundary of a convex body $\mathcal{K}$ \label{alg:Approx_Convex}}
\flushleft
 \textbf{input:} Intersection oracle for the convex body $\mathcal{K}$.\\
 \textbf{input:} Oracle for the normal vector of $\mathcal{M} = \partial \mathcal{K}$. \\
 \textbf{input:} Initial point $x\in \mathcal{M}$, initial velocity $v\in \mathcal{T}_x \mathcal{M}$, parameter $\theta > 0$ .\\
 \textbf{output:} $\Delta^\star_{\theta; (x,v)}$, $\psi^\star_{\theta;(x,v)} := (\gamma_{\theta;(x,v)}^\star, \varphi_{\theta;(x,v)}^\star)$.
\begin{algorithmic}[1]
\\ Generate the unique unit vector $u$ pointing into $\mathcal{K}$ at $x$ at an angle $\theta$ to the tangent plane $\mathcal{T}_x\mathcal{M}$, whose projection $\tilde{u}$ onto $\mathcal{T}_x \mathcal{M}$ satisfies $\measuredangle(\tilde{u},v)=0$ (using the normal vector oracle).
\\ Let $\ell := \{su: s \in \mathbb{R}\}$ be the line passing through $x$ in the direction $u$, and let $P$ be the $2$-plane spanned by $v$ and $u$. 
\\ Solve for the other intersection point $x_1^\dagger$ in $\ell \cap \mathcal{M}$ by calling  the intersection oracle. 
\\ Set $\tilde{v}_1$ to be the projection of $u$ onto the 1-dimensional subspace $Q:= \mathbbm{n}^\perp(x_1^\dagger) \cap P$, and set $v_1^\dagger = \frac{\tilde{v}_{1}}{\| \tilde{v}_{1}\|}$.
\\ Set $(\gamma_{\theta;(x,v)}^\star, \varphi_{\theta;(x,v)}^\star) = (x_1^\dagger, v_1^\dagger)$.
\\ Set $\Delta^\star_{\theta; (x,v)} = \|x_1^\dagger - x\|$.
\end{algorithmic}
\end{algorithm}

In this section we will make a slightly strengthened assumption about the curvature.  Towards this end, we define the inner radius of curvature at a point $x \in \partial \mathcal{K}$ to be the radius of the largest sphere in $\mathbb{R}^{d+1}$ that is tangent to $\partial \mathcal{K}$ at $x$ and contained in $\mathcal{K}$.  Similarly, we define the outer radius of curvature to be radius of the smallest sphere that is tangent to $\partial \mathcal{K}$ at $x$ that contains $\mathcal{K}$.
\begin{assumption}\label{assumption:radius}
$\mathcal{M} = \partial \mathcal{K}$ has inner radius of curvature uniformly bounded below by $\frac{1}{\sqrt{\mathfrak{M}_2}}$ and outer radius of curvature uniformly bounded above by $\frac{1}{\sqrt{\mathfrak{m}_2}}$.
\end{assumption}
In particular, the above assumptions on the inner and outer radii of curvature imply our previous assumptions that the sectional curvature is bounded above and below by $\mathfrak{M}_2$ and $\mathfrak{m}_2$, respectively, although the converse is not true.

The following Lemma (Lemma \ref{thm:chord}) relates the geodesic distance on a convex body to the Euclidean distance in the ambient space.  We will use Lemma \ref{thm:chord} to prove the main result of this section (Lemma \ref{thm:approx_convex}).

Before we state Lemma \ref{thm:chord}, we define $\mathrm{dist}_\mathcal{N}(x,y)$ to be the geodesic distance between two points $x,y \in \mathcal{N}$ in a manifold $\mathcal{N}$.  For every set $S = \{x,y\}$ consisting of two distinct points $x,y \in \mathcal{N},$ we also define $\mathrm{dist}_\mathcal{N}(S) := \mathrm{dist}_\mathcal{N}(x,y)$.

\begin{lemma}\label{thm:chord}
Let $q^{(1)}$ and $q^{(2)}$ be two points on the boundary $\partial \mathcal{K}$ of a convex body $\mathcal{K}$, with inner and outer radii of curvature bounded below and above by $\frac{1}{\sqrt{\mathfrak{M}_2}}$ and $\frac{1}{\sqrt{\mathfrak{m}_2}}$, respectively.  Then  
\begin{equation}
\mathrm{dist}_{\partial \mathcal{K}}(q^{(1)},q^{(2)}) \leq 2\uppi \frac{\sqrt{\mathfrak{M}_2}}{\sqrt{\mathfrak{m_2}}} \|q^{(2)} - q^{(1)}\|
\end{equation}
\end{lemma}
\begin{proof}

Let $\ell$ be the line passing through both $q^{(1)}$ and $q^{(2)}$. 
Let $B$ and $B'$ be balls tangent to $\partial \mathcal{K}$ at $q^{(2)}$ of radius $\frac{1}{\sqrt{\mathfrak{M}_2}}$ and $\frac{1}{\sqrt{\mathfrak{m}_2}}$, respectively. By assumption,
\begin{equation}\label{eq:nested}
B \subseteq \mathcal{K} \subseteq B'.
\end{equation}
Since $B$, $\mathcal{K}$, and $B'$ are all convex, we therefore have that (see Figure \ref{fig:chord}),
\begin{equation}\label{eq:b0}
\mathrm{dist}_{\partial \mathcal{K}}(\ell \cap \partial \mathcal{K}) \leq \mathrm{dist}_{\partial \mathcal{K} \cap A}(\ell \cap \partial \mathcal{K})  \leq \mathrm{dist}_{\partial B'}(\ell \cap \partial B') + \mathrm{length}(\ell \cap B') \leq 2\mathrm{dist}_{\partial B'}(\ell \cap \partial B'), 
\end{equation}
where  $A$ is the $2$-plane containing the great circle (i.e. the spherical geodesic) in $B'$ that connects the two points in $\ell \cap \partial B'$.
\begin{figure}
\centering
\begin{tikzpicture}[scale=1.4]

\draw[thick] (-6,-0.7)--(3,1.28);

\begin{scope}
\clip (-1.51,-1.36) circle (3);
\draw[fill={rgb:black,1;white,5}] (-6,-0.7)--(3,1.28)--(-5,5)--cycle;
\end{scope}

\draw[thick] (0.28,0.5) circle (0.42);
\draw[thick] (-1.51,-1.36) circle (3);

\draw [red, thick, domain=-1.118:1.118, samples=100] plot ({sqrt(1-0.8*abs(\x^2))/(1+0.3*\x)},\x);
\draw [red, thick, domain=-1.118:1.118, samples=100] plot ({-sqrt(1-0.8*abs(\x^2))/(1+0.3*\x)+0.04},\x);

\draw[blue,fill=blue] (0.6,0.75) circle (.35ex);
\node[blue] at (0.83,1.1)  {$q^{(2)}$};

\draw[blue,fill=blue] (-0.74,0.44) circle (.35ex) node[above] {$q^{(1)}$};

\node[red] at (0,-1.5) {$\mathcal{K}\cap A$};
\node at (-0.05,-0.05) {$B\cap A$};

\node at (2.5,1.5) {$\ell$};

\node at (-3.8,-3.8) {$B'\cap A$};

\end{tikzpicture}
\caption{An illustration of the convex geometry in the proof of Lemma \ref{thm:chord}. \label{fig:chord}}
\end{figure}
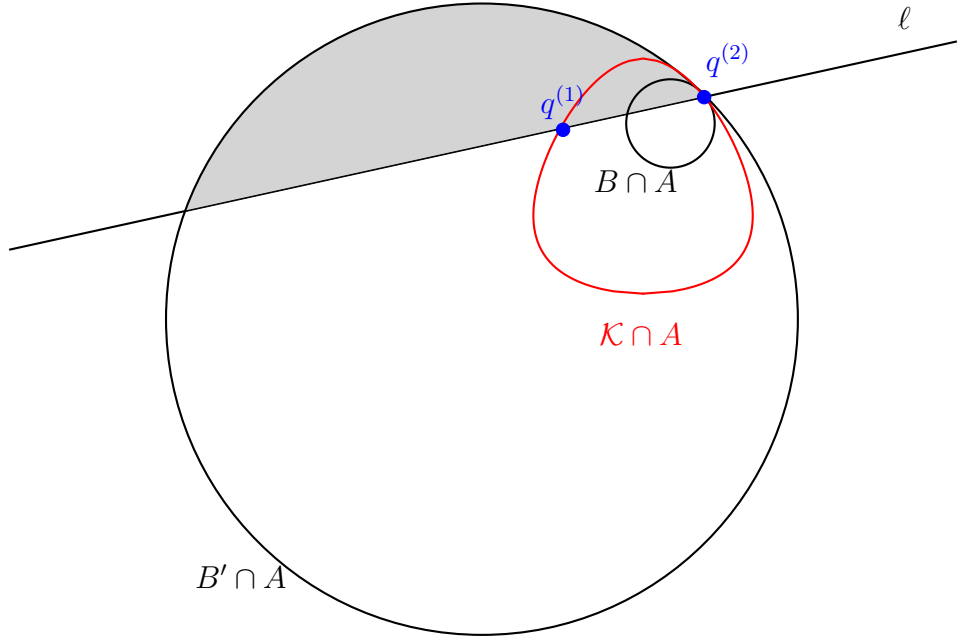

Since the ratio $\frac{\mathrm{dist}_{\partial B'}(\ell \cap \partial B')}{\mathrm{length}(\ell \cap B')}$ is maximized when $\ell \cap B'$ is the diameter of $B'$, it must be true that
\begin{equation}\label{eq:b1}
\mathrm{dist}_{\partial B'}(\ell \cap B') \leq \uppi \mathrm{length}(\ell \cap B').
\end{equation}

Assuming without loss of generality that $q^{(2)}$ is the origin, we note
\begin{equation}\label{eq:b2}
B' = \frac{\sqrt{\mathfrak{M}_2}}{\sqrt{\mathfrak{m}_2}}B,
\end{equation}
so that
\begin{equation}
\frac{\sqrt{\mathfrak{M}_2}}{\sqrt{\mathfrak{m}_2}}(\ell \cap B)=  \ell \cap (\frac{\sqrt{\mathfrak{M}_2}}{\sqrt{\mathfrak{m}_2}} B) = \ell \cap B'.
\end{equation}
Hence,
\begin{equation}\label{eq:b3}
\mathrm{length}(\ell \cap B') = \frac{\sqrt{\mathfrak{M}_2}}{\sqrt{\mathfrak{m}_2}}\mathrm{length}( \ell \cap B) \leq \frac{\sqrt{\mathfrak{M}_2}}{\sqrt{\mathfrak{m}_2}}\mathrm{length}( \ell \cap \partial K) = \frac{\sqrt{\mathfrak{M}_2}}{\sqrt{\mathfrak{m}_2}} \|q^{(2)} - q^{(1)}\|.
\end{equation}

Therefore, combining Equations \ref{eq:b0}, \ref{eq:b1}, and \ref{eq:b3}, we get:
\be \mathrm{dist}_{\partial \mathcal{K}}(q^{(1)},q^{(2)}) &= \mathrm{dist}_{\partial \mathcal{K}}(\ell \cap \partial \mathcal{K})
\leq 2\mathrm{dist}_{\partial B'}(\ell \cap \partial B')\\
&\leq 2 \uppi \mathrm{length}(\ell \cap B')
\leq 2\uppi \frac{\sqrt{\mathfrak{M}_2}}{\sqrt{\mathfrak{m}_2}} \|q^{(2)} - q^{(1)}\|.
\ee
\end{proof}

\begin{lemma}\label{thm:approx_convex}
Let $\mathcal{K}$ be a convex body and $\mathcal{M} = \partial \mathcal{K}$ be its boundary. Set $\alpha = 5\uppi \frac{\mathfrak{M}_2}{\mathfrak{m_2}}$ and $\beta = 10\frac{\sqrt{\mathfrak{M}_2}}{\sqrt{\mathfrak{m}_2}}$. Fix $(x,v) \in \mathcal{M}^\circ$. For $0 < \theta  < \frac{1}{70} \frac{\mathfrak{m}_2}{\mathfrak{M}_2}$ that satisfies $\beta \cdot \sqrt{\mathfrak{M}_2} \cdot \theta \cdot \Delta^\star_{\theta; (x,v)} <1$, let $\psi_{\theta;(x,v)}^\star := (\gamma_{\theta;(x,v)}^\star, \varphi_{\theta;(x,v)}^\star)$ and $\Delta^\star_{\theta; (x,v)}$ be generated by Algorithm \ref{alg:Approx_Convex}.  Then 
\begin{equation} \label{IneqApproxConvexLastMinute1}
\mathrm{dist}(\gamma_{(x,v)}(\Delta^\star_{\theta; (x,v)}), \gamma_{\theta;(x,v)}^\star) \leq  \alpha \cdot \theta \cdot \Delta^\star_{\theta; (x,v)}
\end{equation}
and
\begin{equation} \label{IneqApproxConvexLastMinute2}
\|\varphi_{\theta;(x,v)}^\star - \overline{\varphi_{(x,v)}(\Delta^\star_{\theta; (x,v)})}\| \leq \beta \cdot \sqrt{\mathfrak{M}_2} \cdot \theta \cdot \Delta^\star_{\theta; (x,v)}.
\end{equation}

Moreover, the function $\Delta^\star_{\theta; (x,v)}$ satisfies Assumption \ref{assumption:polynomial}, with
\begin{equation} \label{EqPolynomialAtEnd}
\frac{\theta}{2\sqrt{\mathfrak{M}_2}} \leq \Delta^\star_{\theta; (x,v)} \leq \frac{\theta}{\sqrt{\mathfrak{m}_2}}
\end{equation}
for all $(x,v) \in \mathcal{M}^\circ$.
\end{lemma}
\begin{rem}
The most difficult part of this proof is checking the bound \ref{IneqApproxConvexLastMinute2} on the error in the velocity approximation $\varphi_{\theta;(x,v)}^\star$, which requires carefully relating the intrinsic geometry of $\partial \mathcal{K}$ to its embedding in Euclidean space. We first show that the component of the error {\em orthogonal} to a certain 2-plane $P$ is $\mathcal{O}(\theta^2)$, and thus negligible.    The component of the error {\em contained} in $P$ is dealt with separately via a geometric argument, illustrated in Figure \ref{fig:approx_convex}. The chosen 2-plane $P$ is the plane spanned by the vector $v$ and the normal vector to $\mathcal{K}$ at $x$. These two vectors represent the initial velocity and acceleration of the geodesic path, and so roughly speaking the plane $P$ captures the zeroth- and first-order terms in our approximation of the velocity.
\end{rem}

\begin{proof}
We begin by proving Inequality \ref{EqPolynomialAtEnd}. Let $u$, $P$, Q, $\ell$, $\Delta^\star_{\theta; (x,v)}$, and $\psi_{\theta;(x,v)}^\star$ be as in Algorithm \ref{alg:Approx_Convex}. Recall that Algorithm \ref{alg:Approx_Convex} sets $\Delta^\star_{\theta; (x,v)} = \mathrm{length}(\ell \cap \mathcal{K})$.  Since the tangent sphere to $\partial \mathcal{K}$ at $x$ of radius $\frac{1}{\sqrt{\mathfrak{M}_2}}$ is contained in $\mathcal{K}$, and the tangent sphere to $\partial \mathcal{K}$ at $x$ of radius $\frac{1}{\sqrt{\mathfrak{m}_2}}$ contains $\mathcal{K}$ (Figure \ref{fig:chord}), and $0 < \theta < \frac{1}{70}$, we have
\begin{equation}\label{eq:sine_linearization}
\frac{\theta}{2\sqrt{\mathfrak{M}_2}} \leq \frac{1}{\sqrt{\mathfrak{M}_2}}\sin(\theta) \leq \Delta^\star_{\theta; (x,v)}  \leq \frac{1}{\sqrt{\mathfrak{m}_2}}\sin(\theta)\leq \frac{\theta}{\sqrt{\mathfrak{m}_2}}.
\end{equation}
Here we used the fact that
\be \label{eq:sin_half}
\frac{1}{2}z \leq \sin(z) \qquad \textrm{ for }0 \leq z \leq \frac{\pi}{2},
\ee
a fact we will use again later in this proof.
Hence, the function $\Delta^\star_{\theta; (x,v)}$ satisfies Assumption \ref{assumption:polynomial}. This completes the proof of Inequality \ref{EqPolynomialAtEnd}.\\

Next, we prove Inequality \ref{IneqApproxConvexLastMinute1}. Define $(x_1,v_1) := \psi_{(x,v)}(\Delta^\star_{\theta; (x,v)})$, and $(x_1^\dagger,v_1^\dagger) := \psi_{\theta;(x,v)}^\star$. By Assumption \ref{assumption:radius},  the principal curvatures of $\mathcal{M}$ are bounded above by $\sqrt{\mathfrak{M}_2}$, so  $\|\gamma''_{(x,v)}(t)\| \leq \sqrt{\mathfrak{M}_2}$ for all $x,v \in \mathcal{M}^\circ$ and $t \in \mathbb{R}$.  Therefore, since $\gamma_{(x,v)}(t) = x + vt + \int_0^t \int_0^\tau \gamma_{(x,v)}''(s) \mathrm{d}s \mathrm{d}\tau$, it must be true that
\be \label{eq:tangent_space_error}
&\qquad \quad \qquad \| x+vt - \gamma_{(x,v)}(t)\| = \left\|x+vt-\left(x+vt +\int_0^t \int_0^\tau \gamma_{(x,v)}''(s) \mathrm{d}s \mathrm{d}\tau \right)\right\| \\
&=\left\|\int_0^t \int_0^\tau \gamma_{(x,v)}''(s) \mathrm{d}s \mathrm{d}\tau \right\|
\leq \int_0^t \int_0^\tau  \|\gamma_{(x,v)}''(s) \| \mathrm{d}s \mathrm{d}\tau
\leq \int_0^t \int_0^\tau \sqrt{\mathfrak{M}_2} \mathrm{d}s \mathrm{d}\tau
=  \frac{1}{2}\sqrt{\mathfrak{M}_2} \cdot t^2
\ee
for every $t \geq 0$. 

Let $\Delta'$ be the unique number such that $\langle x + v \Delta' - x_1^\dagger, v\rangle = 0$, and let $\mathfrak{b}:= \|(x+v\Delta') - x_1^\dagger \|$ (Figure \ref{fig:Position_Error}).  Then the line segment connecting $(x+v\Delta' )$ and $x_1^\dagger$ is orthogonal to $v$.  Then we have
\be \label{eq:right_triangle}
\mathfrak{b} &= \sin(\theta) \times \Delta^\star_{\theta; (x,v)},\\
\Delta' &= \cos(\theta) \times \Delta^\star_{\theta; (x,v)}.
\ee
Therefore,
\be \label{eq:right_triangle_error}
\qquad \qquad  \|x_1 - x_1^\dagger\|
&\leq  \|x_1 - (x+v\Delta^\star_{\theta; (x,v)}) \| \\ &\qquad + \|(x+v\Delta^\star_{\theta; (x,v)}) - (x+v\Delta') \| + \|(x+v\Delta') - x_1^\dagger \|\\
&=  \|x_1 - (x+v\Delta^\star_{\theta; (x,v)}) \| + \|(\Delta^\star_{\theta; (x,v)} - \Delta')v \| + \mathfrak{b}\\
&\stackrel{{\scriptsize \textrm{Eq. }} \ref{eq:right_triangle}}{=}  \|\gamma_{(x,v)}(\Delta^\star_{\theta; (x,v)}) - (x+v\Delta^\star_{\theta; (x,v)}) \|\\ &\qquad + \left|\Delta^\star_{\theta; (x,v)} -  \cos(\theta) \times \Delta^\star_{\theta; (x,v)}\right| \times \|v \| + \sin(\theta) \times \Delta^\star_{\theta; (x,v)}\\
&\stackrel{{\scriptsize \textrm{Eq. }} \ref{eq:tangent_space_error}}{\leq} \frac{1}{2}\sqrt{\mathfrak{M}_2} \cdot (\Delta^\star_{\theta; (x,v)})^2 \\ & \qquad + \Delta^\star_{\theta; (x,v)}\times (1- \cos(\theta)) \times 1 + \sin(\theta) \times \Delta^\star_{\theta; (x,v)}\\
&= \frac{1}{2}\sqrt{\mathfrak{M}_2} \cdot (\Delta^\star_{\theta; (x,v)})^2 + \Delta^\star_{\theta; (x,v)}\times [1- \cos(\theta) + \sin(\theta)]\\
&\leq \frac{1}{2}\sqrt{\mathfrak{M}_2} \cdot (\Delta^\star_{\theta; (x,v)})^2 + \Delta^\star_{\theta; (x,v)}\times 2\theta
\ee
since $1- \cos(\theta) + \sin(\theta) \leq 2 \theta$ for $\theta \geq 0$.

\begin{figure}[t]
\begin{center}
\includegraphics[trim={50mm 155mm 70mm 0mm}, clip, scale=0.5]{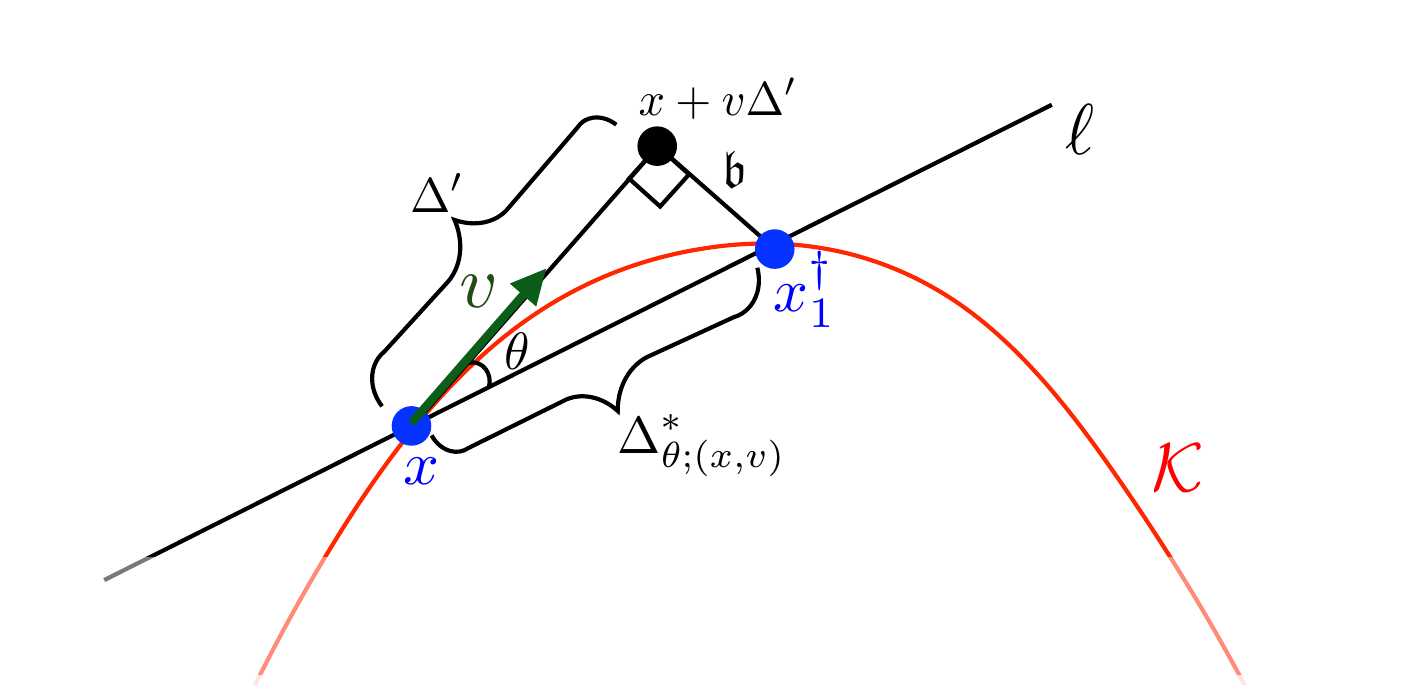} 
\end{center}
\caption{Bounding the error in the position in the proof of Lemma \ref{thm:approx_convex}. \label{fig:Position_Error}}
\end{figure}

Hence, 
\be
\mathrm{dist}(&\gamma_{(x,v)}(\Delta^\star_{\theta; (x,v)}), \gamma_{\theta;(x,v)}^\star) = \mathrm{dist}(x_1, x_1^\dagger)
\stackrel{{\scriptsize \textrm{Lemma }} \ref{thm:chord}}{\leq}  2\uppi \frac{\sqrt{\mathfrak{M}_2}}{\sqrt{\mathfrak{m_2}}} \|x_1 - x_1^\dagger\| \\
&\stackrel{{\scriptsize \textrm{Eq.}} \ref{eq:right_triangle_error}}{\leq}  2\uppi \frac{\sqrt{\mathfrak{M}_2}}{\sqrt{\mathfrak{m_2}}} \left[\frac{1}{2}\sqrt{\mathfrak{M}_2} \times (\Delta^\star_{\theta; (x,v)})^2 + \Delta^\star_{\theta; (x,v)}\times 2\theta\right] \\
&\stackrel{{\scriptsize \textrm{Eq.}} \ref{eq:sine_linearization}}{\leq}  2\uppi \frac{\sqrt{\mathfrak{M}_2}}{\sqrt{\mathfrak{m_2}}} \left[\frac{1}{2}\sqrt{\mathfrak{M}_2} \times \Delta^\star_{\theta; (x,v)} \times \frac{\theta}{\sqrt{\mathfrak{m_2}}} + \Delta^\star_{\theta; (x,v)}\times 2\theta\right] \\
&\leq  2\uppi \frac{\sqrt{\mathfrak{M}_2}}{\sqrt{\mathfrak{m_2}}} \left[\frac{1}{2}\frac{\sqrt{\mathfrak{M}_2}}{\sqrt{\mathfrak{m}_2}} \times \Delta^\star_{\theta; (x,v)} \times \theta + 2\frac{\sqrt{\mathfrak{M}_2}}{\sqrt{\mathfrak{m}_2}} \times \Delta^\star_{\theta; (x,v)}\times \theta\right] \\
&=  5\uppi \left(\frac{\sqrt{\mathfrak{M}_2}}{\sqrt{\mathfrak{m_2}}}\right)^2 \times \Delta^\star_{\theta; (x,v)} \times \theta
=  \alpha \times \Delta^\star_{\theta; (x,v)} \times \theta,
\ee
where the last inequality holds since $\frac{\sqrt{\mathfrak{M}_2}}{\sqrt{\mathfrak{m}_2}} \geq 1$.  This completes the proof of Inequality \ref{IneqApproxConvexLastMinute1}.

Finally, we prove Inequality \ref{IneqApproxConvexLastMinute2}. For $(x,v) \in \mathcal{M}^\circ$, define $H_x(v) := \|\gamma_{(x,v)}''(0)\|$.  By Assumption \ref{assumption:radius},
\be \label{eq:geodesic_second_derivative}
\sqrt{\mathfrak{m}_2}\leq H_x(v) \leq \sqrt{\mathfrak{M}_2}.
\ee
and
\be \label{eq:normal_line_derivative}
\|\mathrm{D}_{\mathsf{u}} \mathbbm{n} (y)\| \leq \sqrt{\mathfrak{M}_2}
\ee
 at every $(x,v), (y,{\mathsf{u}}) \in \mathcal{M}^{\circ}$. 
 
Recall that $P$ is the plane spanned by the vectors $\mathbbm{n}(x)$ and $u$ (this 2-plane is also the 2-plane of curvature associated with the curve $\gamma_{(x,v)}(t)$ at $t=0$), and let $P^\perp$ be its orthogonal complement.  Denote by $\mathcal{P}$ and $\mathcal{P}^\perp$ the operators on $\mathbb{R}^{d+1}$ that project a vector onto $P$ and $P^\perp$, respectively.

Define $(x(t),v(t)):=\psi_{(x,v)}(t)$.  Since $\mathcal{P}^\perp v(0)=0$ and $\mathcal{P}^\perp \mathbbm{n}(\gamma_{(x,v)}(0)) = \mathcal{P}^\perp \mathbbm{n}(x) = 0$, we have:
\be \label{eq:b4}
\qquad \qquad \|\mathcal{P}^\perp v(t)\| 
&=\left \|\int_0^t \mathcal{P}^\perp H_{x(\tau)}(v(\tau)) \mathbbm{n}(\gamma_{(x,v)}(\tau)) \mathrm{d}\tau \right\|\\
&\stackrel{{\scriptsize \textrm{Eq. }} \ref{eq:geodesic_second_derivative}}{\leq} \int_0^t  \left \| \sqrt{\mathfrak{M}_2} \mathcal{P}^\perp \mathbbm{n}(\gamma_{(x,v)}(\tau)) \right\| \mathrm{d}\tau\\
&= \sqrt{\mathfrak{M}_2} \int_0^t  \left \| \int_0^\tau \mathrm{D}_{v(s)} \mathcal{P}^\perp \mathbbm{n}(\gamma_{(x,v)}(s)) \mathrm{d}s +  \mathcal{P}^\perp \mathbbm{n}(\gamma_{(x,v)}(0)) \right\| \mathrm{d}\tau\\
&\leq \sqrt{\mathfrak{M}_2} \int_0^t \int_0^\tau  \left \| \mathrm{D}_{v(s)} \mathcal{P}^\perp \mathbbm{n}(\gamma_{(x,v)}(s)) \right\| \mathrm{d}s \mathrm{d}\tau\\
&\stackrel{{\scriptsize \textrm{Eq. }} \ref{eq:normal_line_derivative}}{\leq}  \sqrt{\mathfrak{M}_2} \int_0^t \int_0^\tau  \sqrt{\mathfrak{M}_2} \mathrm{d}s \mathrm{d}\tau
=\frac{(\sqrt{\mathfrak{M}_2})^2 }{2} t^2.
\ee

Therefore,
Equation \ref{eq:b4} gives
\begin{equation}\label{eq:b5}
\|\mathcal{P}^{\perp} v_1\| = \|\mathcal{P}^\perp v(\Delta^\star_{\theta; (x,v)})\| \leq \frac{(\sqrt{\mathfrak{M}_2})^2 }{2} (\Delta^\star_{\theta; (x,v)})^2.
\end{equation}

Now define $\overline{v}_1(t) := \phi(t;v_1;x_1,x_1^\dagger)$ to be the parallel transport of $v_1$ from $x_1$ to $x_1^\dagger$ along the distance-minimizing unit-speed geodesic $\omega(t)  \equiv \omega(t;x_1,x_1^\dagger)$. Let $t^\sharp := \mathrm{dist}(x_1,x_1^\dagger)$ and let $\overline{v_1} := \overline{v}_1(t^\sharp)$.  Then 
\[\frac{\mathrm{d}}{\mathrm{d}t} \overline{v_1(t)} = \left \langle \overline{v_1(t)},  \frac{\mathrm{d}}{\mathrm{d}t} \mathbbm{n}(\omega(t))\right\rangle \mathbbm{n}(\omega(t)).\]
Therefore, by a similar calculation to that in Equation \ref{eq:b4},
\be \label{eq:b4b}
&\qquad \qquad \qquad \qquad \quad   \|\mathcal{P}^\perp v_1-\mathcal{P}^\perp \overline{v}_1\| =\bigg \|\int_0^{t^\sharp} \mathcal{P}^\perp \frac{\mathrm{d}}{\mathrm{d}\tau} \overline{v_1(\tau)}  \mathrm{d}\tau \bigg \|\\
&= \int_0^{t^\sharp} \left\| \mathcal{P}^\perp  \left \langle \overline{v_1(\tau)},  \frac{\mathrm{d}}{\mathrm{d}\tau} \mathbbm{n}(\omega(\tau))\right\rangle \mathbbm{n}(\omega(\tau))  \right\| \mathrm{d}\tau\\
&\leq \int_0^{t^\sharp} \left\|\frac{\mathrm{d}}{\mathrm{d}\tau} \mathbbm{n}(\omega(\tau))\right\| \times \left\| \mathcal{P}^\perp  \mathbbm{n}(\omega(\tau))  \right\| \mathrm{d}\tau
\stackrel{{\scriptsize \textrm{Eq. }} \ref{eq:normal_line_derivative}}{\leq} \int_0^{t^\sharp}  \sqrt{\mathfrak{M}_2} \left \|\mathcal{P}^\perp \mathbbm{n}(\omega(\tau)) \right\| \mathrm{d}\tau\\
&= \sqrt{\mathfrak{M}_2} \int_0^{t^\sharp}  \left \| \int_0^\tau \frac{\mathrm{d}}{\mathrm{d}s} \mathcal{P}^\perp \mathbbm{n}(\omega(s)) \mathrm{d}s +  \mathcal{P}^\perp \mathbbm{n}(\omega(0)) \right\| \mathrm{d}\tau\\
&\leq \sqrt{\mathfrak{M}_2} \int_0^{t^\sharp}  \left \| \int_0^\tau \frac{\mathrm{d}}{\mathrm{d}s} \mathcal{P}^\perp \mathbbm{n}(\omega(s)) \mathrm{d}s\right\| +  \left\| \mathcal{P}^\perp \mathbbm{n}(x_1) \right\| \mathrm{d}\tau\\
&= \sqrt{\mathfrak{M}_2} \int_0^{t^\sharp}  \left \| \int_0^\tau \frac{\mathrm{d}}{\mathrm{d}s} \mathcal{P}^\perp \mathbbm{n}(\omega(s)) \mathrm{d}s\right\|
+  \bigg\| \int_0^{\Delta^\star_{\theta; (x,v)}} \mathrm{D}_{v(s)} \mathcal{P}^\perp \mathbbm{n}(\gamma_{(x,v)}(s)) \mathrm{d}s +  \mathcal{P}^\perp \mathbbm{n}(\gamma_{(x,v)}(0)) \bigg\| \mathrm{d}\tau\\
&\stackrel{{\scriptsize \textrm{Eq. }} \ref{eq:normal_line_derivative}}{\leq} \sqrt{\mathfrak{M}_2} \int_0^{t^\sharp} \int_0^\tau \sqrt{\mathfrak{M}_2} \mathrm{d}s +  \int_0^{\Delta^\star_{\theta; (x,v)}}\sqrt{\mathfrak{M}_2} \mathrm{d}s \mathrm{d}\tau
= (\sqrt{\mathfrak{M}_2})^2 \left[\frac{1}{2}(t^\sharp)^2 +  t^\sharp \times \Delta^\star_{\theta; (x,v)}\right].
\ee

Since $0 < \theta  < \frac{1}{70} \frac{\mathfrak{m}_2}{\mathfrak{M}_2}$ by assumption, Inequality \ref{IneqApproxConvexLastMinute1} implies that
\be \label{eq:t_sharp}
t^\sharp \leq \Delta^\star_{\theta; (x,v)}.
\ee
Therefore, by Equation \ref{eq:b4b},
\begin{equation}\label{eq:parallel_transport}
\|\mathcal{P}^\perp v_1-\mathcal{P}^\perp \overline{v_1}\| \leq \frac{3}{2}(\sqrt{\mathfrak{M}_2})^2(\Delta^\star_{\theta; (x,v)})^2 \leq 2 (\sqrt{\mathfrak{M}_2})^2(\Delta^\star_{\theta; (x,v)})^2.
\end{equation}

Hence, Equations \ref{eq:b5} and \ref{eq:parallel_transport} imply that
\be \label{eq:b4c}
&\qquad \qquad \|\mathcal{P}^{\perp} \overline{v_1}\|
\leq \|\mathcal{P}^{\perp} v_1\| + \|\mathcal{P}^\perp v_1-\mathcal{P}^\perp \overline{v_1}\|\\
&\leq \frac{(\sqrt{\mathfrak{M}_2})^2 }{2} \cdot (\Delta^\star_{\theta; (x,v)})^2 + 2(\sqrt{\mathfrak{M}_2})^2 \cdot (\Delta^\star_{\theta; (x,v)})^2
= \frac{5}{2}(\sqrt{\mathfrak{M}_2})^2 \cdot (\Delta^\star_{\theta; (x,v)})^2.
\ee

Let $\mathsf{b}:= \measuredangle(\overline{v_1},P)$.  Since $\mathsf{b}$ is the angle between a vector and a plane, $0\leq \mathsf{b} \leq \frac{\pi}{2}$, so $\sin(\mathsf{b})\geq \frac{1}{2}\mathsf{b}$ by Equation \ref{eq:sin_half}.  Thus,
\be\label{eq:b_sin}
\qquad \frac{1}{2}\mathsf{b} &\leq \sin(\mathsf{b}) =\frac{\|\mathcal{P}^\perp \overline{v_1}\|}{\|\overline{v_1}\|} =\frac{\|\mathcal{P}^\perp \overline{v_1}\|}{1} \stackrel{{\scriptsize \textrm{Eq. }} \ref{eq:b4c}}{\leq}  \frac{5}{2}(\sqrt{\mathfrak{M}_2})^2 \cdot (\Delta^\star_{\theta; (x,v)})^2,
\ee
and so
\be \label{eq:b_angle}
\mathsf{b}\leq 5(\sqrt{\mathfrak{M}_2})^2 \cdot (\Delta^\star_{\theta; (x,v)})^2.
\ee

Since $\theta < \frac{1}{70} \frac{\mathfrak{m}_2}{\mathfrak{M}_2}$ by assumption, we have
\be \label{eq:Delta_constant}
\sqrt{\mathfrak{M}_2} \Delta^\star_{\theta; (x,v)} \stackrel{{\scriptsize \textrm{Eq. }}\ref{EqPolynomialAtEnd}}{\leq} \frac{\sqrt{\mathfrak{M}_2}}{\sqrt{\mathfrak{m}_2}}\theta < \frac{1}{70} \frac{\sqrt{\mathfrak{m}_2}}{\sqrt{\mathfrak{M}_2}} \leq \frac{1}{70}.
\ee

 Since the radius of curvature of $\partial \mathcal{K}$ is uniformly bounded below by $\frac{1}{\sqrt{\mathfrak{M}_2}}$, this implies
\be \label{eq:one}
&\measuredangle (\mathbbm{n}^\perp(x), \mathbbm{n}^\perp(x_1^\dagger)) = \measuredangle (\mathbbm{n}(x), \mathbbm{n}(x_1^\dagger)) \leq \sqrt{\mathfrak{M}_2}\mathrm{dist}(x,x_1^\dagger)\\
& \stackrel{{\scriptsize \textrm{Lemma }} \ref{thm:chord}}{\leq}  \sqrt{\mathfrak{M}_2} \times 2\uppi \frac{\sqrt{\mathfrak{M}_2}}{\sqrt{\mathfrak{m}_2}} \|x_1^\dagger - x\|
= 2\uppi \sqrt{\mathfrak{M}_2}  \frac{\sqrt{\mathfrak{M}_2}}{\sqrt{\mathfrak{m}_2}} \Delta^\star_{\theta; (x,v)}.
\ee
Also, since $\mathbbm{n}(x)$ is contained in the plane $P$, we have that
\be \label{eq:two}
\measuredangle(P,\mathbbm{n}^\perp(x)) 
 = \frac{\pi}{2}.
\ee
Now,
\be \label{eq:projection_angle}
& \measuredangle \left(\mathrm{proj}_P\mathbbm{n}(x), \mathrm{proj}_P\mathbbm{n}(x_1^\dagger)\right) = \measuredangle \left(\mathbbm{n}(x), \mathrm{proj}_P\mathbbm{n}(x_1^\dagger)\right)\\
&\leq  \measuredangle \left(\mathbbm{n}(x), \mathbbm{n}(x_1^\dagger)\right) + \measuredangle \left(\mathbbm{n}(x_1^\dagger), \mathrm{proj}_P\mathbbm{n}(x_1^\dagger)\right)\\
&\leq \measuredangle \left(\mathbbm{n}(x), \mathbbm{n}(x_1^\dagger)\right) + \measuredangle \left(\mathbbm{n}(x_1^\dagger), \mathbbm{n}(x)\right)
=2 \measuredangle \left(\mathbbm{n}(x), \mathbbm{n}(x_1^\dagger)\right),
\ee
where the first equality is true since $\mathbbm{n}(x) \in P$ implies that $\mathrm{proj}_P\mathbbm{n}(x) = \mathbbm{n}(x)$, the first inequality is the triangle inequality for angles between vectors, and the second inequality is true by Inequality \ref{eq:projection_minimizes_angle} and the fact that $\mathbbm{n}(x) \in P$. 

Define $\mathsf{c} := \measuredangle(P,\mathbbm{n}^\perp(x_1^\dagger))$.  Equation \ref{eq:projection_angle} implies that
\be \label{eq:projection_angle2}
 & \quad \measuredangle(\mathrm{proj}_P\mathbbm{n}(x_1^\dagger),\mathbbm{n}(x_1^\dagger)) \leq \measuredangle \left(\mathrm{proj}_P\mathbbm{n}(x_1^\dagger), \mathrm{proj}_P\mathbbm{n}(x)\right)\\ & \qquad \qquad+ \measuredangle(\mathrm{proj}_P\mathbbm{n}(x) ,\mathbbm{n}(x))+ \measuredangle(\mathbbm{n}(x),\mathbbm{n}(x_1^\dagger))\\
&\stackrel{{\scriptsize \textrm{Eq. \ref{eq:projection_angle}}}}{\leq} 2\measuredangle(\mathbbm{n}(x),\mathbbm{n}(x_1^\dagger))+ \measuredangle(\mathrm{proj}_P\mathbbm{n}(x) ,\mathbbm{n}(x))+ \measuredangle(\mathbbm{n}(x),\mathbbm{n}(x_1^\dagger))\\
&= 3\measuredangle(\mathbbm{n}(x),\mathbbm{n}(x_1^\dagger))+ \measuredangle(\mathrm{proj}_P\mathbbm{n}(x) ,\mathbbm{n}(x))
= 3\measuredangle(\mathbbm{n}(x),\mathbbm{n}(x_1^\dagger))+ 0,
\ee
so
\be \label{eq:three}
& \qquad \qquad \qquad  \mathsf{c} = \measuredangle(P,\mathbbm{n}^\perp(x_1^\dagger))
=\frac{\pi}{2} - \measuredangle(\mathrm{proj}_P\mathbbm{n}(x_1^\dagger),\mathbbm{n}(x_1^\dagger))\\
&\stackrel{{\scriptsize \textrm{Eq. \ref{eq:projection_angle2}}}}{\geq} \frac{\pi}{2} -  3\measuredangle(\mathbbm{n}(x),\mathbbm{n}(x_1^\dagger))
\stackrel{{\scriptsize \textrm{Eq. \ref{eq:one}}}}{\geq} \frac{\pi}{2} - 3\times2\uppi \sqrt{\mathfrak{M}_2}  \frac{\sqrt{\mathfrak{M}_2}}{\sqrt{\mathfrak{m}_2}} \Delta^\star_{\theta; (x,v)}
\stackrel{{\scriptsize \textrm{Eq. \ref{eq:Delta_constant}}}}{>} \frac{\pi}{4}.
\ee

Therefore, $\mathbbm{n}^\perp(x_1^\dagger)$ does not contain the 2-plane $P$, so $\mathbbm{n}^\perp(x_1^\dagger) \cap P $ must be a 1-dimensional subspace. Define $\mathcal{Q}$ and $\mathcal{Q}^\perp$  to be the projection operators onto the one-dimensional subspace $Q := \mathbbm{n}^\perp(x_1^\dagger) \cap P $ and its orthogonal complement, respectively.  We will now show that $\measuredangle(\mathcal{Q}\overline{v_1},  v_1^\dagger)=0$. 
\be \label{eq:b4'}
& \qquad \| v(t) -  v(0)\| =\left \|\int_0^t H_{x(\tau)}(v(\tau)) \mathbbm{n}(\gamma_{(x,v)}(\tau)) \mathrm{d}\tau \right\|\\
&\stackrel{{\scriptsize \textrm{Eq. }} \ref{eq:geodesic_second_derivative}}{\leq} \int_0^t  \left \| \sqrt{\mathfrak{M}_2} \mathbbm{n}(\gamma_{(x,v)}(\tau)) \right\| \mathrm{d}\tau
= \int_0^t  \left \| \sqrt{\mathfrak{M}_2} \right\| \times 1 \mathrm{d}\tau
= \sqrt{\mathfrak{M}_2} t.
\ee

Also,
\be \label{eq:b4b'}
&\qquad \qquad \| v_1-\overline{v}_1\| =\bigg \|\int_0^{t^\sharp} \frac{\mathrm{d}}{\mathrm{d}\tau} \overline{v_1(\tau)}  \mathrm{d}\tau \bigg \|
= \int_0^{t^\sharp} \left\|  \left \langle \overline{v_1(\tau)},  \frac{\mathrm{d}}{\mathrm{d}\tau} \mathbbm{n}(\omega(\tau))\right\rangle \mathbbm{n}(\omega(\tau))  \right\| \mathrm{d}\tau\\
&\stackrel{{\scriptsize \textrm{Eq. }} \ref{eq:normal_line_derivative}}{\leq} \int_0^{t^\sharp}  \sqrt{\mathfrak{M}_2} \left \| \mathbbm{n}(\omega(\tau)) \right\| \mathrm{d}\tau
= \int_0^{t^\sharp}  \sqrt{\mathfrak{M}_2} \times 1 \mathrm{d}\tau
= \sqrt{\mathfrak{M}_2} t^\sharp
\stackrel{{\scriptsize \textrm{Eq. }}\ref{eq:t_sharp}}{\leq} \sqrt{\mathfrak{M}_2} \Delta^\star_{\theta; (x,v)}.
\ee
Hence,
\be \label{eq:Palgorithm3}
\qquad \|v-\overline{v_1}\| &\leq \|v-v_1\| + \|v_1-\overline{v_1}\|
= \|v(0) - v(\Delta^\star_{\theta; (x,v)})\| + \|v_1-\overline{v_1}\|\\
&\stackrel{{\scriptsize \textrm{Eq. }} \ref{eq:b4'}, \ref{eq:b4b'}}{\leq} \sqrt{\mathfrak{M}_2} \Delta^\star_{\theta; (x,v)} + \sqrt{\mathfrak{M}_2} \Delta^\star_{\theta; (x,v)}
\stackrel{{\scriptsize \textrm{Eq. \ref{eq:Delta_constant}}}}{\leq} \frac{3}{70}.
\ee
But $\|\overline{v_1}\| = \|v\|=1$, so $\overline{v_1}$, $v$ and $v-\overline{v_1}$ form an iscoceles triangle whose height bisects the angle  $\measuredangle(v,\overline{v_1})$.  Hence,
\be \label{eq:Palgorithm4}
\sin\left(\frac{1}{2}\measuredangle(v,\overline{v_1})\right) &= \frac{1}{2}\|v-\overline{v_1}\|
 \stackrel{{\scriptsize \textrm{Eq. \ref{eq:Palgorithm3}}}}{\leq} \frac{1}{2} \times \frac{3}{70}.
\ee
 Since $\|v-\overline{v_1}\| \leq \frac{3}{70} < 1= \|\overline{v_1}\| = \|v\|=1$, $\measuredangle(v,\overline{v_1})$ must be the smallest angle of this iscoceles triangle, so $\measuredangle(v,\overline{v_1}) \leq \frac{\pi}{3}$. Therefore, since $0\leq \measuredangle(v,\overline{v_1}) \leq \frac{\pi}{3}$,
\be 
\frac{1}{2} \measuredangle(v,\overline{v_1}) \stackrel{{\scriptsize \textrm{Eq. \ref{eq:sin_half}}}}{\leq} 2\sin \left(\frac{1}{2}\measuredangle(v,\overline{v_1})\right) \stackrel{{\scriptsize \textrm{Eq. \ref{eq:Palgorithm4}}}}{\leq} \frac{3}{70}.
\ee
Hence,
\be \label{eq:Palgorithm5}
\measuredangle(v,\overline{v_1}) \leq \frac{6}{70}.
\ee

Now, any unit normal vector of the line $P \cap \mathbbm{n}^\perp(\cdot)$ that is also contained in the plane $P$ is equal (up to a scalar factor) to $\mathrm{proj}_P\mathbbm{n}(\cdot)$.  Therefore,
\be \label{eq:Palgorithm1}
\qquad \qquad \qquad \qquad \measuredangle &\left(P \cap \mathbbm{n}^\perp(x),P \cap \mathbbm{n}^\perp(x_1^\dagger)\right) = \measuredangle \left(\mathrm{proj}_P\mathbbm{n}(x), \mathrm{proj}_P\mathbbm{n}(x_1^\dagger)\right)\\
&\stackrel{{\scriptsize \textrm{Eq. \ref{eq:projection_angle}}}}{\leq} 2 \measuredangle \left(\mathbbm{n}(x), \mathbbm{n}(x_1^\dagger)\right)
= 2\measuredangle \left(\mathbbm{n}^\perp(x), \mathbbm{n}^\perp(x_1^\dagger)\right).
\ee
Therefore, 
\be \label{eq:Palgorithm2}
&\qquad \qquad \measuredangle(u,P \cap \mathbbm{n}^\perp(x_1^\dagger)) \leq \measuredangle(u, P \cap \mathbbm{n}^\perp(x)) + \measuredangle \left(P \cap \mathbbm{n}^\perp(x),P \cap \mathbbm{n}^\perp(x_1^\dagger)\right)\\
&\stackrel{{\scriptsize \textrm{Eq. \ref{eq:Palgorithm1}}}}{\leq} \measuredangle(u, P \cap \mathbbm{n}^\perp(x)) + 2\measuredangle \left(\mathbbm{n}^\perp(x), \mathbbm{n}^\perp(x_1^\dagger)\right)
= \theta + 2\measuredangle \left(\mathbbm{n}^\perp(x), \mathbbm{n}^\perp(x_1^\dagger)\right)\\
&\stackrel{{\scriptsize \textrm{Eq. \ref{eq:one}}}}{\leq}  \theta + 2 \times 2\uppi \sqrt{\mathfrak{M}_2}  \frac{\sqrt{\mathfrak{M}_2}}{\sqrt{\mathfrak{m}_2}} \Delta^\star_{\theta; (x,v)}
\stackrel{{\scriptsize \textrm{Eq. \ref{eq:Delta_constant}}}}{\leq} \theta + 2\times \frac{2 \pi}{70}.
\ee

But $v_1^\dagger$ is the (normalized) projection of $u$ onto $P \cap \mathbbm{n}^\perp(x_1^\dagger)$, so $\measuredangle(u, v_1^\dagger) = \measuredangle(u,P \cap \mathbbm{n}^\perp(x_1^\dagger))$.  Hence,
\be \label{eq:Palgorithm6}
&\measuredangle(v, v_1^\dagger) \leq \measuredangle(v, u) + \measuredangle(u, v_1^\dagger)
 = \theta + \measuredangle(u,P \cap \mathbbm{n}^\perp(x_1^\dagger))\\
& \stackrel{{\scriptsize \textrm{Eq. \ref{eq:Palgorithm2}}}}{\leq}  \theta  + \theta + 2 \times \frac{2 \pi}{70}
\leq \frac{3}{10},
\ee
since $0 < \theta  < \frac{1}{70} \frac{\mathfrak{m}_2}{\mathfrak{M}_2} \leq \frac{1}{70} \times 1$ by assumption.
Therefore,
\be \label{eq:pi_tenth}
\measuredangle(\overline{v}_1, v_1^\dagger) &\leq \measuredangle(\overline{v}_1,v) + \measuredangle(v, v_1^\dagger)
  \stackrel{{\scriptsize \textrm{Eq. \ref{eq:Palgorithm5}, \ref{eq:Palgorithm6}}}}{\leq} \frac{6}{70}  + \frac{3}{10}
< \frac{\pi}{5}.
\ee


Now, observe that by Equation \ref{eq:t_sharp},
\be
\mathrm{dist}(x,x_1^\dagger) \leq \mathrm{dist}(x,x_1) + \mathrm{dist}(x_1,x_1^\dagger)\leq \Delta^\star_{\theta; (x,v)} + t^\sharp
\stackrel{{\scriptsize \textrm{Eq. }}\ref{eq:t_sharp}}{\leq} 2\Delta^\star_{\theta; (x,v)}.
\ee

Define $\mathsf{a}:= \measuredangle(\overline{v_1},\mathbbm{n}^\perp(x_1^\dagger)\cap P)$.  Since $\sin(\mathsf{a}) \times \sin(\mathsf{c}) = \sin(\mathsf{b})$ (figure \ref{fig:approx_convex}), we have
\be \label{eq:sin_a}
&\qquad \sin(\mathsf{a}) = \frac{\sin(\mathsf{b})}{\sin(\mathsf{c})}
\stackrel{{\scriptsize \textrm{Eq. }}\ref{eq:three}}{\leq}  \frac{\sin(\mathsf{b})}{ \sin(\frac{\pi}{4})}
\stackrel{{\scriptsize \textrm{Eq. }}\ref{eq:b_sin}}{\leq} \sqrt{2}\times 5(\sqrt{\mathfrak{M}_2})^2 \cdot (\Delta^\star_{\theta; (x,v)})^2.
\ee
\begin{figure}[t]
\begin{center}
\includegraphics[trim={33mm 58mm 105mm 70mm}, clip, scale=0.5]{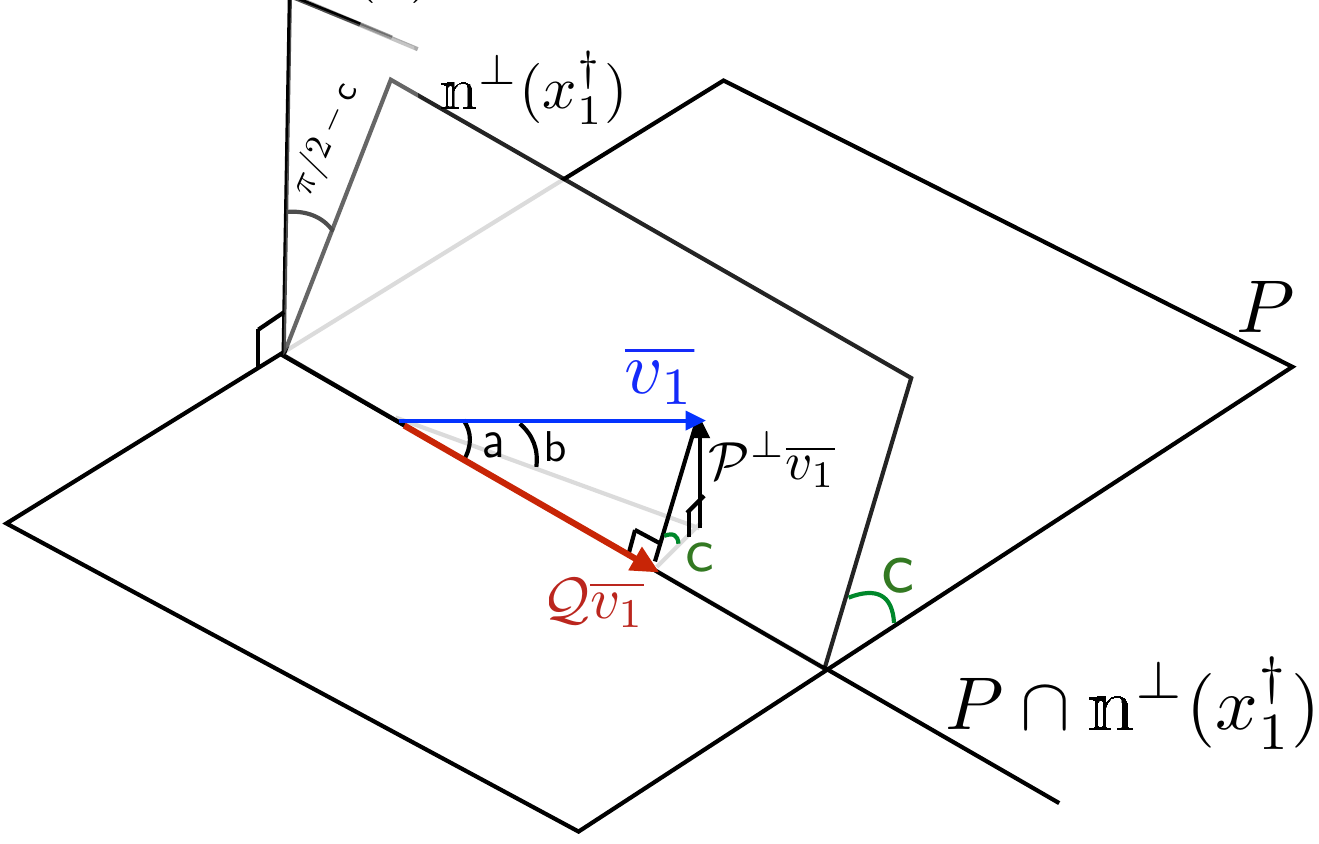} 
\end{center}
\caption{An illustration of the planes and vectors used to bound the error in the velocity in the proof of Lemma \ref{thm:approx_convex}. \label{fig:approx_convex}}
\end{figure}

Since $\mathsf{a}$ is the angle between a vector and a subspace, $0\leq \mathsf{a} \leq \frac{\pi}{2}$, and so
\be
\frac{\|\mathcal{Q}\overline{v_1}\|}{\|\overline{v_1}\|} &= \cos(\mathsf{a})
\geq 1-\sin^2(\mathsf{a})
\stackrel{{\scriptsize \textrm{Eq. }}\ref{eq:sin_a}}{\geq} 1 -  (5\sqrt{2}(\sqrt{\mathfrak{M}_2})^2 \cdot (\Delta^\star_{\theta; (x,v)})^2)^2.
\ee
Since $\|\overline{v_1}\|=1$, we get that
\be \label{eq:Q_projection}
\|\mathcal{Q}\overline{v_1}\| \geq  1 -  (5\sqrt{2}(\sqrt{\mathfrak{M}_2})^2 \cdot (\Delta^\star_{\theta; (x,v)})^2)^2.
\ee

But $v_1^\dagger \in P \cap \mathbbm{n}^\perp(x_1^\dagger)$ and $\mathcal{Q} \overline{v}_1 \in P \cap \mathbbm{n}^\perp(x_1^\dagger)$, so both $v_1^\dagger$ and $\mathcal{Q} \overline{v}_1$ lie in the same 1-dimensional subspace, $P \cap \mathbbm{n}^\perp(x_1^\dagger)$.  Since we have by Equation \ref{eq:pi_tenth} that $\measuredangle(\overline{v}_1, v_1^\dagger) < \frac{\pi}{5}$
 and both $v_1^\dagger$ and $\mathcal{Q} \overline{v}_1$ lie in the same 1-dimensional subspace, it must be true that
\be \label{eq:angle_vector}
\measuredangle(\mathcal{Q}\overline{v_1},  v_1^\dagger)=0.
\ee

Therefore,
\be \label{eq:email1}
&\qquad \qquad \|v_1^\dagger - \mathcal{Q}\overline{v_1}\| \stackrel{{\scriptsize \textrm{Eq. }}\ref{eq:angle_vector}}{=} \|v_1^\dagger\| - \|\mathcal{Q}\overline{v_1}\|\\
&\stackrel{{\scriptsize \textrm{Eq. }}\ref{eq:Q_projection}}{\leq} 1- ( 1 -  (5\sqrt{2}(\sqrt{\mathfrak{M}_2})^2 \cdot (\Delta^\star_{\theta; (x,v)})^2)^2)
=  (5\sqrt{2}(\sqrt{\mathfrak{M}_2})^2 \cdot (\Delta^\star_{\theta; (x,v)})^2)^2 .
\ee

Also,
by Equation \ref{eq:sin_a} we have,
\be \label{eq:email2}
\|\mathcal{Q}&\overline{v_1}- \overline{v_1}\| = \|\mathcal{Q}^\perp \overline{v_1}\|
= \|\overline{v_1}\| \times \sin(\mathsf{a})\\
&= 1\times \sin(\mathsf{a})
 \leq 5\sqrt{2}(\sqrt{\mathfrak{M}_2})^2 \cdot (\Delta^\star_{\theta; (x,v)})^2 .
\ee

Therefore, by Equations \ref{eq:email1} and \ref{eq:email2},
\be \label{eq:email3}
&\qquad \qquad \qquad \|v_1^\dagger - \overline{v_1}\| \leq \|v_1^\dagger - \mathcal{Q}\overline{v_1}\| + \|\mathcal{Q}\overline{v_1}- \overline{v_1}\|\\
&\leq (5\sqrt{2}(\sqrt{\mathfrak{M}_2})^2 \cdot (\Delta^\star_{\theta; (x,v)})^2)^2 + 5\sqrt{2}(\sqrt{\mathfrak{M}_2})^2 \cdot (\Delta^\star_{\theta; (x,v)})^2
\leq  10(\sqrt{\mathfrak{M}_2})^2 \cdot (\Delta^\star_{\theta; (x,v)})^2,
\ee
where the last inequality holds since $(\sqrt{\mathfrak{M}_2})^2 \cdot (\Delta^\star_{\theta; (x,v))})^2< \frac{1}{70^2}$ by Equation \ref{eq:Delta_constant}.  Hence,
\be
\|\varphi_{\theta;(x,v)}^\star &- \overline{\varphi_{(x,v)}(\Delta^\star_{\theta; (x,v)})}\| = \|v_1^\dagger - \overline{v_1}\|
\stackrel{{\scriptsize \textrm{Eq. }} \ref{eq:email3}}{\leq}  10(\sqrt{\mathfrak{M}_2})^2 \cdot(\Delta^\star_{\theta; (x,v)})^2\\
&\stackrel{{\scriptsize \textrm{Eq. }} \ref{eq:sine_linearization}}{\leq}  10(\sqrt{\mathfrak{M}_2})^2 \cdot \frac{\theta}{\sqrt{\mathfrak{m}_2}} \cdot \Delta^\star_{\theta; (x,v)}
= \beta \cdot \sqrt{\mathfrak{M}_2} \cdot \theta \cdot \Delta^\star_{\theta; (x,v)}.
\ee
\end{proof}

We conclude with a bound on the total computational effort that is required to sample approximately from the uniform distribution on $\partial \mathcal{K}$ with error less than any fixed $\epsilon > 0$:

\begin{thm}\label{thm:mixing_convex}
Fix $\epsilon > 0$, let $\alpha, \beta$ be defined as in Lemma \ref{thm:approx_convex}, and let $I(\epsilon), \theta(\epsilon)$ be defined as in Equation \ref{EqThetaIDef}. Let $\{X_{i}^{\theta(\epsilon)}\}_{i \in \mathbb{N}}$ be the Markov chain defined in Equation \ref{EqDefGeodesicWalkApprox}, with the oracle $\bigstar$ provided by Algorithm \ref{alg:Approx_Convex}. Then the distribution $\Pi_{I(\epsilon)+1}^{\theta(\epsilon)}$ of $X_{I(\epsilon)+1}^{\theta(\epsilon)}$ satisfies 
\begin{equation} \label{IneqFinalSilly1}
W_{d}(\Pi_{I(\epsilon)+1}^{\theta(\epsilon)}, \Pi) \leq \epsilon,
\end{equation}
and the intersection oracle  and normal vector oracle in Algorithm \ref{alg:Approx_Convex} are called no more than
\begin{equation} \label{IneqFinalSilly2}
N(\epsilon) = \mathcal{O}\bigg( (\frac{\mathfrak{M}_{2}}{\mathfrak{m}_{2}})^3  \times \frac{\log(\frac{2\uppi}{\epsilon\sqrt{\mathfrak{m}_2}})}{\epsilon \cdot \sqrt{\mathfrak{M}_2}} \bigg)
\end{equation}
times.

\end{thm}

\begin{proof}

Inequality \ref{IneqFinalSilly1} is an application of Lemma \ref{IneqNumStepsAppr}. To prove Inequality \ref{IneqFinalSilly2}, by Theorem \ref{thm:e1}, 
\be
N(\epsilon) &\leq \frac{\log(\frac{\epsilon \, \sqrt{\mathfrak{m}_2} }{2 \uppi})}{\log(\cos(\frac{\uppi \sqrt{\mathfrak{m}_2}}{2 \sqrt{\mathfrak{M}_2}}))} \times \bigg[ \lceil \frac{2\uppi \, \sqrt{\mathfrak{M}_2}}{2\sqrt{\mathfrak{M}_2} \, \theta(\epsilon)} \rceil + \mathcal{O}(\log(\frac{\mathfrak{M}_2}{\mathfrak{m}_2}\times \frac{\sqrt{\mathfrak{M}_2}}{\theta(\epsilon)})) \bigg]  \\
&= \mathcal{O}\bigg( \frac{  \log(\frac{\epsilon \, \sqrt{\mathfrak{m}_2} }{2 \uppi}) }{\log(\cos(\frac{\uppi \sqrt{\mathfrak{m}_2}}{2 \sqrt{\mathfrak{M}_2}}))} \times \bigg[ \frac{1}{\epsilon \cdot \sqrt{\mathfrak{M}_2}  (1 - \cos(\frac{\uppi \sqrt{\mathfrak{m}_2}}{2 \sqrt{\mathfrak{M}_2}})) } \times \big[1 + \frac{\uppi}{2}\alpha + (\frac{\uppi}{2})^2\beta\big] \bigg]\bigg) \\
&= \mathcal{O}\bigg( \frac{  \log(\frac{\epsilon \, \sqrt{\mathfrak{m}_2} }{2 \uppi}) }{\log(\cos(\frac{\uppi \sqrt{\mathfrak{m}_2}}{2 \sqrt{\mathfrak{M}_2}}))} \times\\
& \bigg[ \frac{1}{\epsilon \cdot \sqrt{\mathfrak{M}_2}  (1 - \cos(\frac{\uppi \sqrt{\mathfrak{m}_2}}{2 \sqrt{\mathfrak{M}_2}})) } \times \big[1 + \frac{\uppi}{2} \cdot 5\uppi (\frac{\sqrt{\mathfrak{M}_2}}{\sqrt{\mathfrak{m_2}}})^2 + (\frac{\uppi}{2})^2\cdot 10 \frac{\sqrt{\mathfrak{M}_2}}{\sqrt{\mathfrak{m}_2}}\big] \bigg] \bigg) \\
&= \mathcal{O}\bigg( \frac{  \log(\frac{\epsilon \, \sqrt{\mathfrak{m}_2} }{2 \uppi}) }{\log(\cos(\frac{\uppi \sqrt{\mathfrak{m}_2}}{2 \sqrt{\mathfrak{M}_2}}))} \times \frac{1 + \frac{\uppi}{2} \cdot 5\uppi (\frac{\sqrt{\mathfrak{M}_2}}{\sqrt{\mathfrak{m_2}}})^2 + (\frac{\uppi}{2})^2\cdot 10\frac{\sqrt{\mathfrak{M}_2}}{\sqrt{\mathfrak{m}_2}}}{\epsilon \cdot \sqrt{\mathfrak{M}_2}  (1 - \cos(\frac{\uppi \sqrt{\mathfrak{m}_2}}{2 \sqrt{\mathfrak{M}_2}})) } \bigg)\\
&= \mathcal{O}\bigg( (\frac{\mathfrak{M}_2}{\mathfrak{m}_2})^3  \times \frac{\log(\frac{2 \uppi}{\epsilon \, \sqrt{\mathfrak{m}_2} })}{\epsilon \cdot \sqrt{\mathfrak{M}_2}} \bigg),
\ee
where the first inequality is due to Inequality \ref{IneqPenConc2} and Lemma \ref{thm:approx_convex}, and the remaining lines are obtained by the definitions of $\theta(\epsilon)$, $\alpha$ and $\beta$.
\end{proof}

\begin{rem}\label{rem:smooth}
If the smoothness of the curvature of $\mathcal{M}$ is uniformly bounded below (i.e., if the derivative of the curvature is uniformly bounded above), we can replace the ratio $\frac{\mathfrak{M}_2}{\mathfrak{m}_2}$ appearing in Lemmas \ref{thm:chord} and \ref{thm:approx_convex} by a constant $\mathfrak{M}_3$ that depends only on a lower bound on the smoothness of the curvature, not on the curvature itself, so that we get $\alpha = 2\uppi \mathfrak{M}_3$ and $\beta = \sqrt{\mathfrak{M}_3}$.  The number of intersection and normal line oracle calls is then $\mathcal{O}^*( (\frac{\mathfrak{M}_2}{\mathfrak{m}_2})^2 \cdot \mathfrak{M}_3 \cdot (\epsilon \, \sqrt{\mathfrak{M}_2})^{-1})$, which is quadratic in the curvature ratio $\frac{\mathfrak{M}_2}{\mathfrak{m}_2}$. \end{rem}

\section{Discussion and Future Work} \label{sec:discussion}

We give informal discussions of three natural questions left open by this paper.

\subsection{Sampling from convex polytopes with the geodesic walk}\label{sec:discussion_polytope} In this paper, we presented the geodesic walk, gave bounds on its mixing properties for general manifolds $\mathcal{M}$ with bounded positive curvature, and analyzed a simple efficient implementation of the walk in the special case that $\mathcal{M} = \partial \mathcal{K}$ is the boundary of a convex set $\mathcal{K}$. However, it is clear that our bounds can be very poor for some natural manifolds, even in this special case. In this section, we discuss a natural open question: how strong is the assumption that $\partial \mathcal{K}$ has bounded positive curvature, and how can this assumption be weakened in practice?

We first note that the assumption of \textit{nonnegative} curvature is very weak. Recall that the \textit{Alexandrov curvature} is a generalization of the usual notion of sectional curvature to more general metric spaces, and that the boundary of a convex body in $\mathbb{R}^{d}$ always has \textit{nonnegative} Alexandrov curvature (see \textit{e.g.} \cite{alexandrov2005convex} for a definition of Alexandrov curvature and survey of relevant results). However, the assumption that the curvature is \textit{bounded} is much stronger: it is straightforward to check that the boundary of any polytope will have Alexandrov curvature that is not bounded away from either 0 or $+ \infty$.

Fortunately, it is possible to obtain efficient samples from the boundaries of convex bodies that do not satisfy our assumptions by using appropriate pre- and post-processing steps. We now give an algorithm that allows us to replace the task of sampling from the boundary of a convex body $\partial \mathcal{K}$ with the task of sampling from a related ``rounded" body $\partial \mathcal{K}'$ whose sectional curvature is bounded from above by $\mathcal{O}(d^{2})$. The construction is based on essentially the same intuition as the pre- and post-processing described in Section 5 of \cite{vempala2005geometric} for sampling from the interior of a convex set, though the required analysis is slightly more delicate. We first require some simple definitions:

\begin{defn}[$\epsilon$-Thickening and Projection]
For fixed $\epsilon > 0$ and convex body $\mathcal{K} \subset \mathbb{R}^{d}$, we define $\mathcal{K}_{\epsilon} = \{ x \in \mathbb{R}^{d} \, : \, \inf_{y \in \mathcal{K}} \| x - y \| \leq \epsilon \}$ and $\mathcal{M}_{\epsilon} = \partial \mathcal{K}_{\epsilon}$. 

We define the projection-like map $f_{\epsilon} \, : \, \mathcal{M}_{\epsilon} \mapsto \mathcal{M} \equiv \partial \mathcal{K}$ by 
\be 
f_{\epsilon}(x) = \mathrm{argmin}_{y \in \mathcal{M}} \| x - y \|.
\ee 
\end{defn}

We define $H_{\epsilon}(x)$ to be the magnitude of the determinant of the Hessian of $f_{\epsilon}$ at $x$, when it exists. We note:

\begin{lemma} \label{LemmaHessianUpper}
With notation as above,
\be 
0 \leq H_{\epsilon}(x) \leq 1
\ee 
for all $x \in \mathcal{M}_{\epsilon}$.
\end{lemma}

\begin{proof}
This is an immediate consequence of the fact that $f_{\epsilon}$ is a contraction; see \cite{58556} for a proof of this fact.
\end{proof}

As a consequence of this lemma, for all $\epsilon > 0$ the following is a valid rejection-sampling algorithm for sampling from $\mathcal{M}$ if you know how to sample from $\mathcal{M}_{\epsilon}$:

\begin{algorithm}[H]
\caption{Simple Rejection Sampling Algorithm \label{AlgRej}}
\flushleft
\textbf{parameters}: Convex set $\mathcal{K}$ and parameters $\epsilon > 0$, $N \in \mathbb{N}$.\\
\textbf{output:} Samples $X_{1},\ldots,X_{N}$ from the uniform measure on $\mathcal{M} \equiv \partial \mathcal{K}$.
\begin{algorithmic}[1]
\For{$i = 1, 2, \ldots, N$}
\\ Sample $Z_{i} \sim \mathrm{Unif}(\mathcal{M}_{\epsilon})$. 
\\ Propose $Y_{i} = f_{\epsilon}(Z_{i})$. 
\\ Accept the sample $Y_{i}$ and set $X_{i} = Y_{i}$ with probability $H_{\epsilon}(Z_{i})$. Otherwise, reject the sample and go back to Step 2. 
\EndFor
\end{algorithmic}
\end{algorithm}

The rejection rate of this algorithm is given by
\be \label{EqRejectP1}
\mathbb{P}[\mathrm{reject}] = \frac{\mathrm{Area}(\mathcal{M})}{\mathrm{Area}(\mathcal{M}_{\epsilon})}.
\ee 

By Crofton's formula, for $\epsilon = \frac{1}{20 \, d}$ we have 
\be \label{EqRejectP2}
\mathrm{Area}(\mathcal{M}_{\epsilon}) \leq \mathrm{Area}(\mathcal{M}) + \frac{1}{2}.
\ee 

Assuming without loss of generality that $\mathrm{Area}(\mathcal{M}) \geq 1$ after appropriate rescaling and choosing $\epsilon = \frac{1}{20d}$, we have by formulas \eqref{EqRejectP1} and \eqref{EqRejectP2} that the rejection probability of Algorithm \ref{AlgRej} is uniformly bounded away from 1:
\be \label{EqRejectP3}
\mathbb{P}[\mathrm{reject}] \leq \frac{2}{3}.
\ee 
Thus, in order to be able to sample from the uniform distribution on $\mathcal{M}$, it is enough to be able to sample from the uniform distribution on  $\mathcal{M}_{\epsilon}$ as long as $\epsilon < \frac{1}{20d}$. We have gained something important in this replacement: while $\mathcal{M}$ may have unbounded sectional curvature, the sectional curvature of $\mathcal{M}_{\epsilon}$ is at most $\epsilon^{-2}$. To see this last fact, consider a point $x \in \mathcal{M}_\epsilon$ and let $y \in \mathcal{M}$ satisfy $\| x - y \| = \epsilon$ (such a nearest point exists by the definition of $\mathcal{M}_{\epsilon}$). Then the ball $B_{\epsilon}(y)$ of radius $\epsilon$ centered at $y$ is entirely contained in the interior of $\mathcal{M}_\epsilon$, and $x$ is on the boundary of both $B_{\epsilon}(y)$ and $\mathcal{M}_{\epsilon}$.  Therefore, the radius of curvature of $\mathcal{M}_{\epsilon}$ at $x$ must be at least $\epsilon$ (recalling that $\mathcal{M}_\epsilon$ has nonnegative curvature because $\mathcal{M}$ is the boundary of a convex body), and hence the sectional curvature at $x$ must be at most $\epsilon^{-2}$. Since this applies for every point $x \in \mathcal{M}_{\epsilon}$, we conclude that the sectional curvature of $\mathcal{M}_{\epsilon}$ is at most $\epsilon^{-2}$.

Inequality \eqref{EqRejectP3} allows us to conclude that, using Algorithm \ref{AlgRej}, we can always transform the problem of sampling uniformly from a manifold $\mathcal{M}$ with the problem of sampling uniformly from a slightly perturbed manifold $\mathcal{M}_{\epsilon}$ with bounded sectional curvature $\sup_{x \in \mathcal{M}_{\epsilon}} \sup_{u,v \in \mathcal{T}_{x} \mathcal{M}_{\epsilon}} C_{x}(u,v) = \mathcal{O}(d^{2})$.

Further discussion of pre- and post-processing is out of the scope of this paper. We leave two related questions for future work:

\begin{enumerate}
\item Is there another perturbation of the manifold that guarantees an even smaller upper bound on the sectional curvature?
\item Is there another perturbation of the manifold that guarantees a useful \textit{lower} bound on the sectional curvature?
\end{enumerate}

\subsection{Sectional or Ricci Curvature}

Readers familiar with previous work on mixing for walks on manifolds with ``positive curvature," such as \cite{ollivier2009ricci}, may be surprised that our main conditions are stated in terms of the \textit{sectional} curvature rather than the \textit{Ricci} curvature. Roughly speaking, a manifold has positive sectional curvature if it is positively curved ``in all directions," while it has positive Ricci curvature under the weaker condition that it is positively curved ``on average." It is natural to ask: are the main results of the paper true if we replace our bound on the sectional curvature with the weaker analogous assumption on the Ricci curvature?

If we consider Algorithm \ref{alg:Geodesic_Walk} for very small integration time $T$, the answer is yes. Indeed, it is straightforward to check that, in the limit as $T$ goes to 0,  Algorithm \ref{alg:Geodesic_Walk} has the same contraction bound as the ball walk studied in Example 4 of \cite{ollivier2009ricci}. However, choosing $T$ to be very small will generally result in a very inefficient algorithm; it would be much more useful if we could obtain a similar result for $T$ large.

For longer integration times, we do not know the answer to the question in any generality, but mention that some positive results are possible using existing (but more advanced) results from differential geometry. For example, we could find an ``averaged" version of the contraction result in Theorem \ref{thm:contraction} if we could replace references to Rauch's theorem to an analogue that assumed only a positive Ricci curvature, rather than the stronger assumption of positive sectional curvature.  There has been substantial research on finding such analogues, under the name of ``comparison geometry" (see the survey \cite{zhu1997comparison}). We don't summarize these results here, except to note that there exist some positive results that would apply to our algorithm, including those in \cite{dai1995comparison}.

\subsection{Faster mixing time bounds for Hamiltonian Monte Carlo} \label{sec:discussion_HMC} Hamiltonian Monte Carlo (HMC), a Markov chain algorithm that is in many ways analogous to the geodesic walk, is popular in applications such as statistics and machine learning where one wishes to sample from a target distribution $\pi:\mathbb{R}^d \rightarrow \mathbb{R}^+$.  The HMC Markov chain proceeds on $\mathbb{R}^d$ in a similar way as the geodesic walk on a manifold $\mathcal{M}$, except that the geodesic trajectories are replaced by the trajectory of a particle moving according to the laws of classical mechanics in a potential well with potential $U = -\log(\pi)$.  When $\pi$ is a high-dimensional product measure, most HMC trajectories are closely approximated by a geodesic on one of the level sets of $\pi$, due to the concentration of measure phenomenon.  In a similar vein, the paper \cite{seiler2014positive} uses concentration of measure to approximate HMC with geodesic trajectories on a related manifold, called the Jacobi manifold.

Using a different approach, the paper \cite{HMC_logconcave} obtained the best current mixing time bound $\mathcal{O}^*((\frac{M_2}{m_2})^2)$ for HMC on strongly logconcave distributions $\pi$, when the eigenvalues of the Hessian matrix of $\pi$ are uniformly bounded above and below by positive constants $M_2$ and $m_2$, respectively.  As each level set $S$ of $\pi$ has sectional curvature ratio $\frac{\mathfrak{M}_2(S)}{\mathfrak{m}_2(S)} \leq \frac{M_2}{m_2}$, where $\mathfrak{M}_2(S)$ and $\mathfrak{m}_2(S)$ are, respectively, the maximum and minimum values of the sectional curvature of $S$, the mixing time bound $\mathcal{O}^*((\frac{M_2}{m_2})^2)$ obtained in \cite{HMC_logconcave} for HMC on the distribution $\pi$ is slower than the mixing time bound of $\mathcal{O}^*(\frac{\mathfrak{M}_2}{\mathfrak{m}_2})$ we obtained in this paper for the geodesic walk on any of the level sets $S$ of $\pi$.

Since HMC trajectories closely approximate geodesic paths, we conjecture that one should be able to obtain a stronger mixing time bound of $\mathcal{O}^*(\frac{M_2}{m_2})$ for HMC that matches the mixing time bound we obtained in this paper for the geodesic walk on any of the level sets $S$ of $\pi$.  We leave as an open problem whether the methods used in this paper to bound the mixing time of the geodesic walk can be extended to obtain the conjectured mixing time bound of $\mathcal{O}^*(\frac{M_2}{m_2})$ for HMC.

\section{Acknowledgements}
We are grateful to Yin Tat Lee and A.B. Dieker for helpful discussions.  We also thank the anonymous reviewer for helpful comments and suggestions.
\bibliographystyle{plain}
\bibliography{Geodesic}

\end{document}